\documentclass[12pt]{elsarticle}
\graphicspath{ {./figures/} }

\usepackage[a4paper, total={6.5in, 9in}]{geometry}
\usepackage{caption}
\usepackage{amsmath}
\usepackage{amssymb}
\usepackage{amsthm}
\usepackage{mathtools}
\usepackage{graphicx}
\usepackage[percent]{overpic}
\usepackage{empheq}
\usepackage{tcolorbox}
\usepackage{mathrsfs}
\usepackage{pifont}
\usepackage{algorithm}
\usepackage{algpseudocode}
\usepackage{natbib}
\usepackage{lscape}
\usepackage{tabularx}
\usepackage{soul}
\usepackage[table]{xcolor}
\usepackage{changepage, threeparttable}
\usepackage{multirow}
\usepackage{adjustbox}
\usepackage{smartdiagram}
\usepackage{longtable}
\usepackage{enumitem}
\usepackage{sectsty}
\usepackage{subcaption}
\usepackage{booktabs}
\usepackage{rotating}
\usepackage{bbding}

\usepackage{tikz}
\usepackage{pgfplots}
\pgfplotsset{compat=1.18}

\usepackage{hyperref}
\usepackage{cleveref}

\hypersetup{
    colorlinks=true,
    urlcolor=blue,
    linkcolor=blue,
    citecolor=red
}

\newtheorem{theorem}{Theorem}[section]
\newtheorem{lemma}[theorem]{Lemma}
\newtheorem{proposition}[theorem]{Proposition}
\newtheorem{corollary}[theorem]{Corollary}
\theoremstyle{definition}
\newtheorem{definition}[theorem]{Definition}
\newtheorem{example}[theorem]{Example}
\newtheorem{remark}[theorem]{Remark}

\newcommand{\AHI}{\mathrm{AHI}}
\newcommand{\PiAHI}{\Pi_{\AHI}}
\newcommand{\SONC}{\mathcal{C}}
\newcommand{\SDSOS}{\mathrm{SDSOS}}

\begin{document}
\begin{frontmatter}

\title{The AHI family of sum of squares polynomials}

\author[label1]{Shaon Naskar}
\ead{snaskar@isihyd.ac.in}

\author[label1]{Kanchan Rajwar}
\ead{kanchanrajwar1519@gmail.com}

\author[label1]{Sujeet Kumar Singh\corref{cor1}}
\ead{sksinghma209@gmail.com, sujeet@isihyd.ac.in}

\cortext[cor1]{Corresponding author}

\address[label1]{Indian Statistical Institute, Hyderabad Unit, 500007, India}


\begin{abstract}
We introduce a new family of non-negative polynomials, constructed via the arithmetic--harmonic inequality, called AHI polynomials. We derive explicit algebraic conditions for this family and prove that, for AHI polynomials, the cone of non-negative polynomials coincides with the cone of sum of squares (SOS) polynomials. We then study their convexity, showing that although AHI polynomials are generally non-convex, certain monomial substructures are SOS-convex. We further locate the family precisely among the standard non-negativity certificates: every AHI polynomial is simultaneously SOS and a sum of non-negative circuit polynomials (SONC), and the containment in the intersection of these two cones is strict. By closing this family under multiplication, we obtain a cone $\Pi_{\mathrm{AHI}}$ that is, by construction, still SOS, yet we prove that it lies outside both the SONC cone and the smaller SDSOS cone. Moreover, membership in this cone admits a closed-form certificate that does not require solving any semidefinite programs. Finally, we demonstrate the usefulness of these structures in optimization: numerical experiments indicate that exploiting AHI sparsity yields a computation time over 300 times faster than dense SOS relaxations and enables solving high-degree polynomial optimization problems (up to degree 40) that standard methods cannot handle due to computational limits, and a factorized hierarchy for $\Pi_{\mathrm{AHI}}$ decomposes products into independent small subproblems that generic sparsity techniques do not detect.
\end{abstract}


\end{frontmatter}
\makeatletter
\def\ps@pprintTitle{%
 \let\@oddhead\@empty
 \let\@evenhead\@empty
 \def\@oddfoot{}%
 \let\@evenfoot\@oddfoot}
\makeatother

\section{Introduction}\label{sec:intro}

The study of non-negative polynomials and their representation as sum of squares (SOS) is a fundamental topic in real algebraic geometry. A foundational contribution to this field is Hilbert's 1888 paper \cite{hilbert1888darstellung}, which provided a complete characterization of the pairs $(n, 2d)$ of dimension and degree for which every non-negative polynomial is a sum of squares. Hilbert demonstrated that in many cases, there exist non-negative polynomials that cannot be written as sums of squares, although he did not provide an explicit example. This result sparked Hilbert's 17th problem in 1900, asking whether every non-negative polynomial can be expressed as a sum of squares of \textit{rational} functions. \citet{artin1927zerlegung} affirmatively resolved this in 1927. However, the gap between the cone of non-negative polynomials ($\mathcal{P}_{n,2d}$) and the cone of SOS polynomials ($\Sigma_{n,2d}$) remained a subject of intense study.

A pivotal advancement came in 1967 when Motzkin constructed the first explicit example of a non-negative polynomial that is not SOS: the ternary sextic $M(x,y,z) = x^4 y^2 + x^2 y^4 - 3 x^2 y^2 z^2 + z^6$ \citep{motzkin1967arithmetic}. This polynomial, derived by substituting squared monomials into the arithmetic mean--geometric mean (AM-GM) inequality, is non-negative by construction, but resists SOS decomposition. Since then, numerous other examples have been documented (see, e.g., \cite{choi1977extremal, choi1977old, choi1980real, reznick2000some, robinson1973some}). \citet{blekherman2006there} later quantified this gap, showing that for fixed degree, the non-negative cone grows asymptotically much faster than the SOS cone as the dimension increases. Building on Motzkin's insight, \citet{reznick1989forms} introduced ``agiforms,'' a family generalizing Motzkin's construction via the AM-GM inequality, while \citet{iliman2016amoebas} characterized non-negativity for sparse polynomials supported on circuits and introduced the cone of \emph{sums of non-negative circuit polynomials} (SONC), a non-negativity certificate independent of SOS.

The distinction between non-negative polynomials and SOS polynomials is not merely theoretical but has profound computational implications. Determining whether a multivariate polynomial of degree four or higher is non-negative is known to be NP-hard \cite{murty1987some}. In contrast, determining whether a polynomial is a sum of squares can be formulated as a semidefinite programming (SDP) feasibility problem, which can be solved in polynomial time with arbitrary precision \cite{vandenberghe1996semidefinite}. Consequently, SOS relaxations have become the standard tool for approximating non-negative polynomials in optimization and control theory. However, since $\Sigma_{n,2d} \subsetneq \mathcal{P}_{n,2d}$ for most dimensions and degrees, these relaxations are not always exact. It is thus of considerable practical relevance to identify the particular algebraic structures in which these two cones are identical.

The optimization community has shown immense interest in this field following the identification of the fundamental connection between semidefinite programming and SOS polynomials, formally stated as follows.

\begin{theorem}[\citealp{choi1995sums, parrilo2000structured}]\label{thm:sos-gram}
A polynomial $f(x)$, $x \in \mathbb{R}^n$, of degree $2d$ is a sum of squares if and only if there exists a positive symmetric semidefinite matrix $M$ such that
\[
f(x) = v_d(x)^{T} M v_d(x),
\]
where $v_d(x)$ denotes the vector of all monomials in $x$ of degree up to $d$, i.e.,
\[
v_d(x) = \big(1,\, x_{1},\, \ldots,\, x_{n},\, x_{1}x_{2},\, \ldots,\, x_{n}^{d}\big)^{T}.
\]
\end{theorem}

Furthermore, a real multivariate polynomial $f(x)$ is non-negative if $f(x)\ge 0$ for all $x\in\mathbb{R}^n$. If $f:\mathbb{R}^n\to\mathbb{R}$ has degree $d$, then $f(x)=\sum_a f_a x^a$, where $x^a = x_1^{a_1}\cdots x_n^{a_n}$ with $\sum_{i=1}^n a_i\le d$, $a_i\in\mathbb{N}\cup\{0\}$, and $\{f_a\}\in\mathbb{R}^{s(d)}$ is the coefficient vector in the basis
\[
1,x_1,\dots,x_n,x_1^2,x_1x_2,\dots,x_n^2,\dots,x_1^d,\dots,x_n^d.
\]
A polynomial with degree $d$ in $n$ variables has $s(d)$ coefficients, where $s(d)=\binom{n+d}{d}$. A polynomial is homogeneous (or a form) if all its monomials have the same degree. A $d$-degree form $f(x)$ satisfies $f(\lambda x)=\lambda^d f(x)$ and has $\binom{n+d-1}{d}$ coefficients. A polynomial $f(x)$ is a sum of squares if there exist polynomials $f_1(x),\dots,f_r(x)$ such that $f(x)=\sum_{i=1}^r f_i(x)^2$, each such $f_i(x)$ having degree at most $d/2$.

In polynomial optimization, Lasserre's hierarchy \cite{lasserre2001global} employs these SOS relaxations to approximate global minima of non-convex problems \cite{lasserre2009moments,lasserre2015introduction}. Extensive research has been conducted to prove the finite convergence of Lasserre's hierarchy under various conditions \cite{laurent2007semidefinite,nie2014optimality}.

Despite the theoretical convergence of SOS hierarchies, the size of the resulting semidefinite programs increases combinatorially with the number of variables and the degree, often leading to the so-called ``curse of dimensionality.'' For general polynomials, the size of the semidefinite matrix $M$ (also known as the Gram matrix) in Theorem~\ref{thm:sos-gram} scales with $\binom{n+d}{d}$. To mitigate this computational bottleneck, recent research has focused on exploiting specific structures such as symmetry \cite{gatermann2004symmetry}, correlative sparsity \cite{waki2006sums}, and term sparsity \cite{wang2021tssos}. By identifying families of polynomials---such as those arising from fundamental inequalities---that possess inherent structural properties, we can derive specialized SOS decompositions that are significantly less expensive to compute than generic relaxations. The AM-GM inequality has been used extensively to create counterexamples (e.g., the Motzkin form), while the arithmetic mean--harmonic mean (AM-HM) inequality has been comparatively overlooked. Although AM-HM yields a stronger bound than AM-GM, it is rarely used to generate polynomial families.

In this paper, we develop a new class of polynomials with the aid of the classical AM-HM inequality. Surprisingly, for these polynomials, the non-negative cone coincides with the SOS cone. That is, the proposed family exhibits an exact correspondence between non-negativity and the existence of a sum-of-squares representation. We call this family of polynomials the \emph{AHI polynomials}. We then ask where this family sits relative to the SONC certificate, which originates from the AM-GM inequality, and show that although every AHI polynomial is SONC, the family is not closed under multiplication with respect to SONC. Closing it under multiplication produces a new cone $\PiAHI$ that is SOS by construction yet lies outside both SONC and SDSOS, and whose factored structure suggests optimization strategies not captured by existing sparsity techniques.

\paragraph{Organization.}
Section~\ref{sec:family} details the construction of the family $\mathcal{P}$ of AHI polynomials and proves the key theoretical result: a member of $\mathcal{P}$ is non-negative if and only if it is a sum of squares. Section~\ref{sec:convexity} investigates the convexity of these polynomials, proving that they are generally non-convex except for specific monomial structures where they exhibit SOS-convexity. Section~\ref{sec:cones} locates the AHI family inside $\Sigma_{n,2d}\cap\SONC_{n,2d}$, introduces the product cone $\PiAHI$, and establishes its separation from SONC and SDSOS. Section~\ref{sec:numerics} explores the application of AHI polynomials in optimization, demonstrating that the specific sparsity structure allows significantly faster computation times compared to standard SOS relaxations in both constrained and unconstrained settings, and develops a factorized optimization hierarchy for $\PiAHI$. Section~\ref{sec:conclusions} concludes. Two appendices collect a table of explicit SOS decompositions and a technical Hessian computation.

\section{The AHI family of non-negative polynomials}\label{sec:family}

We define a special family of polynomials in $n$ variables as follows:
\begin{equation}\label{eq:family}
\mathcal{P}=\Big\{p: p=
\Big(\sum_{i=1}^n \lambda_i x^{\alpha(i)}\Big) \Big(\sum_{i=1}^n \lambda_i \prod_{{j\neq i}, j=1}^n x^{\alpha(j)}\Big) - c\prod_{i=1}^n x^{\alpha(i)}\Big\},
\end{equation}
where $\lambda_i \geq 0$, $x^{\alpha(i)}$ is a squared monomial in $n$ variables of degree at most $2d$, and $c$ is a real number.

For example, for $n=3$, equation \eqref{eq:family} simplifies to
\[
  (\lambda_1 x^{\alpha(1)} + \lambda_2 x^{\alpha(2)} + \lambda_3 x^{\alpha(3)})(\lambda_1 x^{\alpha(2)+\alpha(3)} + \lambda_2 x^{\alpha(1)+\alpha(3)} + \lambda_3 x^{\alpha(1)+\alpha(2)}) - c\, x^{\alpha(1)+\alpha(2)+\alpha(3)},
\]
which can be rearranged as
\[
\begin{aligned}
 & \lambda_1\lambda_2 x^{2\alpha(2)+\alpha(3)} + \lambda_1\lambda_3 x^{\alpha(2)+2\alpha(3)} + \lambda_1\lambda_2 x^{2\alpha(1)+\alpha(3)} +\lambda_2\lambda_3 x^{\alpha(1)+2\alpha(3)} +
  \lambda_1\lambda_3 x^{2\alpha(1)+\alpha(2)} \\ &+\lambda_2\lambda_3 x^{\alpha(1)+2\alpha(2)}
  + (\lambda_1^2 + \lambda_2^2 + \lambda_3^2 - c) x^{\alpha(1)+\alpha(2)+\alpha(3)}.
  \end{aligned}
\]

\begin{example}\label{ex:ahi1}
Take $\lambda_1=\lambda_2=\lambda_3=1/3$, $c=1$, and choose the squared monomials $x^{\alpha(1)}=x_1^2$, $x^{\alpha(2)}=x_2^2$, $x^{\alpha(3)}=x_3^2$. The resulting AHI polynomial is
\begin{equation}\label{eq:ahi-ex1}
    p_1(x) = \tfrac{1}{9}(x_1^4 x_2^2 + x_2^4 x_1^2 + x_1^4 x_3^2 + x_3^4 x_1^2 + x_2^4 x_3^2 + x_3^4 x_2^2) - \tfrac{2}{3}x_1^2x_2^2x_3^2.
\end{equation}
\end{example}

\begin{example}\label{ex:ahi2}
Taking $\lambda_1=1/2$, $\lambda_2=1/4$, $\lambda_3=1/4$ and the distinct squared monomials $x^{\alpha(1)}=x_1^2 x_2^2$, $x^{\alpha(2)}=x_2^4$, $x^{\alpha(3)}=x_2^2 x_3^4$, we obtain another AHI polynomial
\begin{equation}\label{eq:ahi-ex2}
    p_2(x) = \tfrac{1}{8}(x_1^4 x_2^8 + x_1^2 x_2^{10}) + \tfrac{1}{8}(x_1^4 x_2^6 x_3^4 + x_1^2 x_2^6 x_3^8) + \tfrac{1}{16}(x_2^{10} x_3^4 + x_2^8 x_3^8) - \tfrac{5}{8} x_1^2 x_2^8 x_3^4.
\end{equation}
\end{example}

Further AHI polynomials can be generated in an analogous way. The theoretical framework for building the AHI family is presented in detail in the next subsection.

\subsection{Construction of a subfamily of $\mathcal{P}$}\label{subsec:construction}

The weighted \emph{Arithmetic Mean--Harmonic Mean (AM-HM)} inequality states that for $n$ positive real numbers $a_1,\dots,a_n$ and $n$ non-negative weights $\lambda_1,\dots,\lambda_n$ with $\sum_{i=1}^n \lambda_i = 1$,
\[
  \lambda_1 a_1 + \cdots + \lambda_n a_n
  \;\;\ge\;\;
  \frac{1}{\frac{\lambda_1}{a_1} + \cdots + \frac{\lambda_n}{a_n}}.
\]
With simple algebraic manipulation and rearrangement, this yields
\begin{equation}\label{eq:amhm}
\Big(\sum_{i=1}^n \lambda_i a_i\Big) \Big(\sum_{i=1}^n \lambda_i \prod_{{j\neq i}, j=1}^n a_j\Big) - \prod_{i=1}^n a_i \;\geq\; 0 .
\end{equation}
Expression \eqref{eq:amhm} is valid for all $n \in \mathbb{N}$. It is the particular case of \eqref{eq:family} with $\sum_{i=1}^n \lambda_i = 1$ and $c=1$.

Substituting squared monomials in $n$ non-zero variables $x_1,\dots,x_n$ of degree at most $2d$ in place of the $a_i$ (squared monomials are always non-negative) in \eqref{eq:amhm}, we obtain the AHI polynomial
\begin{equation}\label{eq:ahi-sub}
    p(x_1,\dots,x_n)=\Big(\sum_{i=1}^n \lambda_i x^{\alpha(i)}\Big) \Big(\sum_{i=1}^n \lambda_i \prod_{{j\neq i}, j=1}^n x^{\alpha(j)}\Big) - \prod_{i=1}^n x^{\alpha(i)} .
\end{equation}

\begin{remark}\label{rem:basic}
\hfill
\begin{enumerate}
\item The polynomial formed by substituting squared monomials in \eqref{eq:amhm} is non-negative by construction.
\item Any positive multiple of $p$ is also called an AHI polynomial. Multiples are usually taken to rationalize the denominators of the coefficients.
\end{enumerate}
\end{remark}

\subsection{SOS representation of the expression \texorpdfstring{\eqref{eq:amhm}}{(3)}}\label{subsec:sos-rep}

Here we discuss the SOS representation of the non-negative polynomial \eqref{eq:amhm}. This representation makes the positive semidefiniteness structure explicit, which is what enables the computational savings of Section~\ref{sec:numerics}.

\begin{theorem}\label{thm:sos-rep}
The non-negative polynomial in expression \eqref{eq:amhm} has an SOS representation.
\end{theorem}

\begin{proof}
We claim that
\begin{equation}\label{eq:sos-identity}
\Big(\sum_{i=1}^n \lambda_i a_i\Big) \Big(\sum_{i=1}^n \lambda_i \prod_{{j\neq i}, j=1}^n a_j\Big) - \prod_{i=1}^n a_i \;=\; \sum_{1\leq i<j\leq n} \lambda_i\lambda_j (a_i - a_j)^2\prod_{k\neq i,j}a_k .
\end{equation}
Proving this suffices to establish the existence of an SOS representation for the AHI polynomial. Expanding the left-hand side of \eqref{eq:sos-identity},
\[
\begin{aligned}
&\Big(\sum_{i=1}^n \lambda_i a_i\Big) \Big(\sum_{i=1}^n \lambda_i \prod_{{j\neq i}, j=1}^n a_j\Big) - \prod_{i=1}^n a_i\\
&=\, \lambda_1\lambda_2(a_{1}^2 + a_{2}^2 - 2a_{1}a_{2})\,a_{3}a_{4}\cdots a_{n}
 + \lambda_1\lambda_3(a_{1}^2 + a_{3}^2 - 2a_{1}a_{3})\,a_{2}a_{4}\cdots a_{n} + \cdots \\
&\quad + \lambda_1\lambda_n(a_{1}^2 + a_{n}^2 - 2a_{1}a_{n})\,a_{2}a_{3}\cdots a_{n-1}
 + \lambda_2\lambda_3(a_{2}^2 + a_{3}^2 - 2a_{2}a_{3})\,a_{1}a_{4}\cdots a_{n} + \cdots \\
&\quad + \lambda_2\lambda_n(a_{2}^2 + a_{n}^2 - 2a_{2}a_{n})\,a_{1}a_{3}\cdots a_{n-1} + \cdots
 + \lambda_{n-1}\lambda_n(a_{n-1}^2 + a_{n}^2 - 2a_{n-1}a_{n})\,a_{1}a_{2}\cdots a_{n-2}\\
&=\, \lambda_1\lambda_2 (a_{1}-a_{2})^{2} a_{3}a_{4}\cdots a_{n}
 + \lambda_1\lambda_3(a_{1}-a_{3})^{2} a_{2}a_{4}\cdots a_{n} + \cdots \\
&\quad + \lambda_1\lambda_n(a_{1}-a_{n})^{2} a_{2}a_{3}\cdots a_{n-1}
 + \lambda_2\lambda_3(a_{2}-a_{3})^{2} a_{1}a_{4}\cdots a_{n} + \cdots \\
&\quad + \lambda_2\lambda_n(a_{2}-a_{n})^{2} a_{1}a_{3}\cdots a_{n-1} + \cdots
 + \lambda_{n-1}\lambda_n(a_{n-1}-a_{n})^{2} a_{1}a_{2}\cdots a_{n-2}\\
&=\,\sum_{1\leq i<j\leq n}\lambda_i\lambda_j(a_i - a_j)^2\prod_{k\neq i,j}a_k. \qedhere
 \end{aligned}
\]
\end{proof}

Consequently, each member of the AHI family has an SOS representation: the polynomials \eqref{eq:ahi-sub}, obtained by substituting squared monomials for the $a_i$ in \eqref{eq:sos-identity}, are SOS polynomials on $\mathbb{R}^{n}$. Moreover, every member of \eqref{eq:ahi-sub} has an explicit SOS form with at most $\binom{n}{2}$ square terms, and the monomials appearing inside those squares number at most
\begin{equation}\label{eq:half-support}
|B| \;=\; n(n-1)+1 .
\end{equation}
Indeed, each pair $i<j$ contributes the two monomials $x^{\alpha(i)+\frac12\sum_{k\ne i,j}\alpha(k)}$ and $x^{\alpha(j)+\frac12\sum_{k\ne i,j}\alpha(k)}$, giving $2\binom{n}{2}=n(n-1)$ monomials, together with the single central monomial $\prod_{i}x^{\alpha(i)/2}$. We refer to this set as the \emph{half-support} of the AHI polynomial; it is exactly the Gram basis exploited in Section~\ref{sec:numerics} and reported as $K_{\mathrm{size}}$ in Table~\ref{tab:scalability}. A list of polynomials and their SOS representations is given in \ref{app:sos-table}.

\subsection{SOS representation of the general AHI family \texorpdfstring{\eqref{eq:family}}{(1)}}\label{subsec:general}

Relaxing the requirement $\sum_{i=1}^n \lambda_i = 1$ and taking $c \in \mathbb{R}$ in \eqref{eq:amhm} generates the general AHI family \eqref{eq:family}. Here we prove that all non-negative AHI polynomials in \eqref{eq:family} are SOS. Expanding \eqref{eq:family}, all coefficients $\lambda_i\lambda_j$ are non-negative except the coefficient $(\lambda_1^2 + \cdots +\lambda_n^2 - c)$ of the monomial $x^{\alpha(1) + \cdots + \alpha(n)}$. Non-negativity therefore hinges entirely on this single coefficient. The expanded form of \eqref{eq:family} reads
\begin{equation}\label{eq:expanded}
   \begin{aligned} p(x)=
      &\, \lambda_1\lambda_2 x^{2\alpha(1)+\alpha(3)+\cdots+\alpha(n)} +\lambda_1\lambda_3 x^{2\alpha(1)+\alpha(2)+\alpha(4)+\cdots+\alpha(n)}+\cdots+\lambda_1\lambda_{n} x^{2\alpha(1)+\alpha(2)+\cdots+\alpha(n-1)}\\
      & +\lambda_1\lambda_2 x^{2\alpha(2)+\alpha(3)+\cdots+\alpha(n)} + \lambda_2\lambda_3 x^{\alpha(1)+2\alpha(2)+\alpha(4)+\cdots+\alpha(n)}+\cdots+\lambda_2\lambda_n x^{\alpha(1)+2\alpha(2)+\alpha(3)+\cdots+\alpha(n-1)}\\
     & +\cdots\\
      & + \lambda_1\lambda_n x^{\alpha(2)+\alpha(3)+\cdots+ 2\alpha(n)} + \lambda_2\lambda_n x^{\alpha(1)+\alpha(3)+\cdots +2\alpha(n)}+\cdots + \lambda_{n-1}\lambda_n x^{\alpha(1)+\cdots +\alpha(n-2) + 2\alpha(n)}\\
      & + (\lambda_1^2 + \cdots + \lambda_n^2 - c)\,x^{\alpha(1) + \cdots+\alpha(n)},
   \end{aligned}
\end{equation}
which can be rearranged as
\begin{equation}\label{eq:regrouped}
    p(x)= \sum_{1\leq i < j\leq n}\lambda_i\lambda_j\big(x^{2\alpha(i)} + x^{2\alpha(j)}\big)\prod_{k\neq i,j}x^{\alpha(k)} +  \Big(\sum_{i=1}^{n}\lambda_i^2 - c\Big)\prod_{i=1}^{n}x^{\alpha(i)} .
\end{equation}
We establish the SOS property for the family \eqref{eq:family} via two results: Theorem~\ref{thm:sufficient} gives a sufficient coefficient condition for the SOS property, and Theorem~\ref{thm:necessary} shows the same condition is necessary for non-negativity.

\begin{theorem}\label{thm:sufficient}
Let $p(x) \in \mathcal{P}$ be as defined in \eqref{eq:family}. If the coefficient of the last term in \eqref{eq:regrouped} satisfies
\[ K \;=\; \lambda_1^2 + \cdots +\lambda_n^2 - c \;\ge\; -2\!\!\sum_{1\leq i < j\leq n}\!\! \lambda_i\lambda_j, \]
then $p(x)$ is an SOS polynomial.
\end{theorem}

\begin{proof}
We analyse the coefficient $K$ in \eqref{eq:regrouped}. First, assume $K$ can be written exactly as a sum of negative cross terms,
\[
K = -2\lambda_i\lambda_j \quad \text{or} \quad -2\lambda_i\lambda_j - 2\lambda_k\lambda_l - \cdots .
\]
Then the term $K\prod_{i=1}^{n}x^{\alpha(i)}$ can be distributed to pair with the corresponding positive terms $\lambda_i\lambda_j(x^{2\alpha(i)} + x^{2\alpha(j)})\prod_{k\neq i,j}x^{\alpha(k)}$, and each grouping forms a perfect square:
\[
\lambda_i\lambda_j\big(x^{2\alpha(i)} + x^{2\alpha(j)}\big) - 2\lambda_i\lambda_j x^{\alpha(i)}x^{\alpha(j)} = \lambda_i\lambda_j \big(x^{\alpha(i)} - x^{\alpha(j)}\big)^2 .
\]
Thus the polynomial decomposes into an SOS representation.

Second, consider the general case where $K$ exceeds the sum of all negative cross terms. Put $S_{\mathrm{cross}} = -2\sum_{1\leq i < j\leq n} \lambda_i\lambda_j$, so that $K = S_{\mathrm{cross}} + k$ with $k \ge 0$. From the expansion \eqref{eq:regrouped}, all terms with positive coefficients, when paired with the negative cross terms of $S_{\mathrm{cross}}$, decompose into squares of the form $(x^{\alpha(i)} - x^{\alpha(j)})^2$ multiplied by non-negative monomial factors. If $k > 0$, there remains a term $k \prod_{i=1}^n x^{\alpha(i)}$; since the $x^{\alpha(i)}$ are squared monomials, their product is a square and $k$ is positive, so this residual term is also a square. Therefore $p(x)$ is a sum of squares.
\end{proof}

\begin{theorem}\label{thm:necessary}
Let $p(x) \in \mathcal{P}$. If the coefficient term satisfies
\[ K \;=\; \lambda_1^2 + \cdots +\lambda_n^2 - c \;<\; -2\!\!\sum_{1\leq i < j\leq n} \!\!\lambda_i\lambda_j, \]
then $p(x)$ is not non-negative, and consequently not SOS.
\end{theorem}

\begin{proof}
We construct a point $x^* \in \mathbb{R}^n$ with $p(x^*) < 0$. Let $x^* = (1,\dots,1)^T$. Evaluating \eqref{eq:regrouped} at this point yields
\[
p(1,\dots,1) = \sum_{1\leq i < j\leq n} 2\lambda_i\lambda_j + K .
\]
Substituting $K = -2\sum_{i<j}\lambda_i\lambda_j - \epsilon$ for some $\epsilon > 0$ gives
\[
p(1,\dots,1) = \sum_{1\leq i < j\leq n} 2\lambda_i\lambda_j + \Big( -2\!\!\sum_{1\leq i < j\leq n}\!\! \lambda_i\lambda_j - \epsilon \Big) = -\epsilon < 0 .
\]
Hence $p$ is not positive semidefinite, and by definition it cannot be a sum of squares.
\end{proof}

Combining Theorems~\ref{thm:sufficient} and~\ref{thm:necessary}, we conclude that for the AHI family $\mathcal{P}$, a polynomial $p(x)$ is non-negative if and only if it is SOS. This is a strong assertion: there is no gap between non-negativity and the existence of an SOS representation for polynomials of this family.

\section{Convexity of AHI polynomials}\label{sec:convexity}

In this section we investigate whether the polynomials of the family \eqref{eq:family} are convex. Recall that an $m\times m$ symmetric polynomial matrix $F$ is positive semidefinite, denoted $F(x)\succeq 0$, if $F(x)$ is positive semidefinite for all $x\in \mathbb{R}^n$, equivalently if the scalar polynomial $y^T F(x) y$ is non-negative in $m+n$ variables. A polynomial matrix $F(x)$ is an SOS-matrix if $F(x)=L(x)^T L(x)$ for some polynomial matrix $L(x)$, equivalently if $y^TF(x)y$ is SOS.

A polynomial $p(x)$ is convex if and only if its Hessian $H(x)$, the $n\times n$ symmetric matrix of second-order partial derivatives, is positive semidefinite. Deciding convexity of polynomials is NP-hard already in degree four \citep{ahmadi2013np}. \citet{helton2010semidefinite} proposed SOS-convexity as a tractable algebraic certificate for convexity. To investigate the convexity of AHI polynomials, we first treat the two-variable case of \eqref{eq:regrouped}.

\begin{theorem}\label{thm:nonconvex-2var}
The polynomials generated from \eqref{eq:regrouped} for $1\leq i < j\leq 2$ with general monomials in two variables are non-convex.
\end{theorem}

\begin{proof}
Such polynomials are given by
\begin{equation}\label{eq:2var}
p(x) = \lambda_1\lambda_2\big(x_1^{\alpha_1}x_2^{\beta_1} + x_1^{\alpha_2}x_2^{\beta_2}\big) + (\lambda_1^2 + \lambda_2^2 - c)\,x_1^{\alpha_1 + \alpha_2}x_2^{\beta_1 + \beta_2},
\end{equation}
where $\alpha_i,\beta_i\geq 2$ are even. The Hessian entries are
\begin{equation}\label{eq:hessian-2var}
    \begin{aligned}
    H_{11} =\, & \lambda_1\lambda_2 (2\alpha_1)(2\alpha_1 -1) x_1^{2\alpha_1 - 2}x_2^{2\beta_1} +  \lambda_1\lambda_2 (2\alpha_2)(2\alpha_2 -1) x_1^{2\alpha_2 - 2}x_2^{2\beta_2}  \\
    & + (\lambda_1^2 + \lambda_2^2 - c)(\alpha_1 + \alpha_2)(\alpha_1 + \alpha_2 - 1)x_1^{\alpha_1 + \alpha_2 - 2}x_2^{\beta_1 + \beta_2},\\[2pt]
      H_{22} =\, & \lambda_1\lambda_2 (2\beta_1)(2\beta_1 -1) x_1^{2\alpha_1}x_2^{2\beta_1 - 2} +  \lambda_1\lambda_2 (2\beta_2)(2\beta_2 -1) x_1^{2\alpha_2}x_2^{2\beta_2 - 2}  \\
    & + (\lambda_1^2 + \lambda_2^2 - c)(\beta_1 + \beta_2)(\beta_1 + \beta_2 - 1)x_1^{\alpha_1 + \alpha_2}x_2^{\beta_1 + \beta_2 - 2},\\[2pt]
       H_{12}= H_{21} =\, & \lambda_1\lambda_2 (2\alpha_1)(2\beta_1)x_1^{2\alpha_1 - 1}x_2^{2\beta_1 - 1} + \lambda_1\lambda_2 (2\alpha_2)(2\beta_2)x_1^{2\alpha_2 - 1}x_2^{2\beta_2 - 1}  \\
       & + (\lambda_1^2 + \lambda_2^2 - c)(\alpha_1 + \alpha_2)(\beta_1 + \beta_2)x_1^{\alpha_1 + \alpha_2 - 1}x_2^{\beta_1 + \beta_2 - 1}.
    \end{aligned}
\end{equation}
If $\lambda_1^2 + \lambda_2^2 > c$, the minors $H_{11}$ and $H_{22}$ are SOS. However, evaluating the determinant $\det(H) = H_{11} H_{22} - H_{12}^2$ shows $H_{11} H_{22} - H_{12}^2 < 0$; the detailed calculation is given in Appendix~\ref{app:hessian}. Hence $H$ is not positive semidefinite and the polynomials are non-convex.
\end{proof}

Next we consider the relaxed monomials $x_1^{\alpha_1}x_2^0$ and $x_1^0 x_2^{\alpha_2}$, where $\alpha_1, \alpha_2 \geq 2$ are even.

\begin{theorem}\label{thm:sos-convex}
Bivariate AHI polynomials of even degree of the form
\begin{equation}\label{eq:2var-diag}
     \lambda_1\lambda_2\big(x_1^{2\alpha_1} + x_2^{2\alpha_2}\big) + (\lambda_1^2 + \lambda_2^2 - c)\,x_1^{\alpha_1}x_2^{\alpha_2}
\end{equation}
are SOS-convex whenever they are convex.
\end{theorem}

\begin{proof}
The Hessian entries of \eqref{eq:2var-diag} are
\begin{equation}\label{eq:hessian-diag}
\begin{aligned}
H_{11} &= \lambda_1\lambda_2(2\alpha_1)(2\alpha_1 - 1)x_1^{2\alpha_1 - 2} + (\lambda_1^2 + \lambda_2^2 - c)\alpha_1(\alpha_1 - 1)x_1^{\alpha_1 - 2}x_2^{\alpha_2},\\
H_{22} &= \lambda_1\lambda_2(2\alpha_2)(2\alpha_2 - 1)x_2^{2\alpha_2 - 2} + (\lambda_1^2 + \lambda_2^2 - c)\alpha_2(\alpha_2 - 1)x_1^{\alpha_1}x_2^{\alpha_2 - 2},\\
H_{12}&=H_{21} = (\lambda_1^2 + \lambda_2^2 - c)\alpha_1\alpha_2 x_1^{\alpha_1 - 1}x_2^{\alpha_2 - 1}.
\end{aligned}
\end{equation}
We analyse the quadratic form $z^{T}H(x)z$:
\begin{equation}\label{eq:quadform}
    \begin{aligned}
        z^{T}H(x)z =\, & \lambda_1\lambda_2(2\alpha_1)(2\alpha_1 - 1)x_1^{2\alpha_1 - 2}z_1^{2} + (\lambda_1^2 + \lambda_2^2 - c)\alpha_1(\alpha_1 - 1)x_1^{\alpha_1 - 2}x_2^{\alpha_2}z_1^{2}\\
        &  + \lambda_1\lambda_2(2\alpha_2)(2\alpha_2 - 1)x_2^{2\alpha_2 - 2}z_2^{2} + (\lambda_1^2 + \lambda_2^2 - c)\alpha_2(\alpha_2 - 1)x_1^{\alpha_1}x_2^{\alpha_2 - 2}z_2^{2}\\
        & + 2(\lambda_1^2 + \lambda_2^2 - c)\alpha_1\alpha_2 x_1^{\alpha_1 - 1}x_2^{\alpha_2 - 1}z_1 z_2 .
    \end{aligned}
\end{equation}
For $H(x)$ to be PSD the principal minors must be non-negative, which requires $\lambda_1^2 + \lambda_2^2 > c$. Assume without loss of generality that the terms associated with $\alpha_2$ dominate those associated with $\alpha_1$. Convexity depends on the magnitude of the cross term $2(\lambda_1^2 + \lambda_2^2 - c)\alpha_1\alpha_2$: if this term is sufficiently small, $z^T H(x) z$ decomposes directly into squares; if it is large, it can be grouped with portions of the diagonal terms to form a perfect square. Since any convex configuration of parameters leads to a decomposition into squares, convexity implies SOS-convexity in this case.
\end{proof}

The family \eqref{eq:regrouped} is non-convex already for three variables.

\begin{theorem}\label{thm:nonconvex-3var}
The three-variable polynomial
\begin{equation}\label{eq:3var}
    \lambda_1\lambda_2(x_1^{\alpha_1} + x_2^{\alpha_2})x_3^{\alpha_3} + \lambda_1\lambda_3(x_1^{\alpha_1} + x_3^{\alpha_3})x_2^{\alpha_2} + \lambda_2\lambda_3(x_2^{\alpha_2} + x_3^{\alpha_3})x_1^{\alpha_1} +
    (\lambda_1^2 + \lambda_2^2 + \lambda_3^2 - c)x_1^{\alpha_1}x_2^{\alpha_2}x_3^{\alpha_3}
\end{equation}
is non-convex.
\end{theorem}

\begin{proof}
The argument parallels that of Theorem~\ref{thm:nonconvex-2var}. Computing the Hessian and analyzing the determinant of the leading principal minors shows that, for general $\alpha_i$, the Hessian fails to be positive semidefinite everywhere, the determinant condition failing because of the dominance of cross terms in the minor expansion, exactly as in the inequality derived in \ref{app:hessian}.
\end{proof}

\section{The AHI family inside \texorpdfstring{$\mathrm{SOS}\cap\mathrm{SONC}$}{SOS cap SONC}, and a product extension}\label{sec:cones}

Having established in Section~\ref{sec:family} that every non-negative AHI polynomial is a sum of squares, we now situate the family precisely relative to the two standard non-negativity certificates: SOS and SONC, the cone of sums of non-negative circuit polynomials introduced by \citet{iliman2016amoebas}. We show that every AHI polynomial belongs to both cones, that the containment in their intersection is strict, and that closing the family under multiplication produces a cone that leaves SONC entirely. Throughout, $\SONC_{n,2d}$ denotes the SONC cone of \citep[Definition~1.3]{iliman2016amoebas} and $\Sigma_{n,2d}$ the SOS cone.

\subsection{AHI is a proper subcone of \texorpdfstring{$\Sigma_{n,2d}\cap\SONC_{n,2d}$}{SOS cap SONC}}\label{subsec:ahi-in-both}

\begin{proposition}\label{prop:ahi-in-sonc}
Every non-negative AHI polynomial $p \in \mathcal{P}$, as defined in \eqref{eq:family}, lies in the SONC cone $\SONC_{n,2d}$.
\end{proposition}

\begin{proof}
By Theorem~\ref{thm:sufficient}, a non-negative $p\in\mathcal{P}$ decomposes as
\begin{equation}\label{eq:ahi-sonc-decomp}
p \;=\; \sum_{1\le i<j\le n} \lambda_i\lambda_j\,g_{ij}(x) \;+\; k\prod_{i=1}^n x^{\alpha(i)},
\qquad
g_{ij}(x)=\big(x^{\alpha(i)}-x^{\alpha(j)}\big)^2\!\!\prod_{k\neq i,j}\! x^{\alpha(k)},
\end{equation}
with $k\ge 0$. Fix a pair $i<j$. The polynomial $g_{ij}$ is supported by the three exponents
\[
2\alpha(i)+\!\!\sum_{k\neq i,j}\!\alpha(k),\qquad
2\alpha(j)+\!\!\sum_{k\neq i,j}\!\alpha(k),\qquad
\alpha(i)+\alpha(j)+\!\!\sum_{k\neq i,j}\!\alpha(k),
\]
the third being the midpoint of the first two. Its Newton polytope is therefore a one-dimensional even lattice simplex (a segment) with the remaining support point in its relative interior, so $g_{ij}$ is a circuit polynomial in the sense of \citep[Section~1]{iliman2016amoebas}. Its circuit number is $\Theta_{g_{ij}}=1$ and its inner coefficient is $-2$ after normalization, therefore $g_{ij}$ is non-negative and, being a perfect square times a monomial square, lies in $\SONC_{n,2d}$ by \citep[Theorem~3.8]{iliman2016amoebas}. The residual term $k\prod_i x^{\alpha(i)}$ is a monomial square and hence trivially SONC. Since $\SONC_{n,2d}$ is a convex cone, $p\in\SONC_{n,2d}$.
\end{proof}

\begin{corollary}\label{cor:ahi-both}
Every non-negative AHI polynomial is simultaneously SOS and SONC; that is,
$\AHI \subseteq \Sigma_{n,2d}\cap\SONC_{n,2d}$.
\end{corollary}

\begin{proof}
Immediate from Theorem~\ref{thm:sufficient} and Proposition~\ref{prop:ahi-in-sonc}.
\end{proof}

The containment of Corollary~\ref{cor:ahi-both} is strict and the obstruction is a rigidity in the coefficient pattern rather than in the support.

\begin{proposition}\label{prop:proper}
For $n\ge 4$ the containment is proper: $\AHI \subsetneq \Sigma_{n,2d}\cap\SONC_{n,2d}$.
\end{proposition}

\begin{proof}
By \eqref{eq:regrouped}, the coefficient attached to the pair $\{i,j\}$ in an AHI polynomial equals $\lambda_i\lambda_j$. Consequently the pair coefficients $c_{ij}=\lambda_i\lambda_j$ of any AHI polynomial obey the multiplicative relations
\begin{equation}\label{eq:coeff-rigidity}
c_{ij}\,c_{kl} \;=\; c_{ik}\,c_{jl} \;=\; c_{il}\,c_{jk}
\qquad\text{for all distinct } i,j,k,l,
\end{equation}
which are non-trivial as soon as $n\ge 4$. Now take $n=4$, $\alpha(i)=2e_i$, and consider
\[
q \;=\; 2x_3^2x_4^2\big(x_1^2-x_2^2\big)^2 + x_2^2x_4^2\big(x_1^2-x_3^2\big)^2 + x_2^2x_3^2\big(x_1^2-x_4^2\big)^2
+ x_1^2x_4^2\big(x_2^2-x_3^2\big)^2 + x_1^2x_3^2\big(x_2^2-x_4^2\big)^2 + x_1^2x_2^2\big(x_3^2-x_4^2\big)^2 .
\]
Each summand is a monomial square times a binomial square, so $q\in\Sigma_{4,8}$, and each summand is a non-negative circuit polynomial exactly as in the proof of Proposition~\ref{prop:ahi-in-sonc}, so $q\in\SONC_{4,8}$. However its pair coefficients are $c_{12}=2$ and $c_{34}=c_{13}=c_{24}=c_{14}=c_{23}=1$, whence $c_{12}c_{34}=2\neq 1=c_{13}c_{24}$, violating \eqref{eq:coeff-rigidity}. Hence $q\notin\AHI$.
\end{proof}

\begin{remark}\label{rem:proper-meaning}
Proposition~\ref{prop:proper} shows the AHI family is a genuinely thin slice of $\Sigma_{n,2d}\cap\SONC_{n,2d}$, cut out by the rank-one condition \eqref{eq:coeff-rigidity} on the matrix of pair coefficients. The interesting question is therefore not that $\AHI$ is a strict subcone, which is to be expected of so structured a family, but what happens when the family is closed under multiplication. This is the subject of the next subsection.
\end{remark}

\subsection{The product AHI cone \texorpdfstring{$\PiAHI$}{Pi AHI}}\label{subsec:piahi}

\begin{definition}\label{def:piahi}
Let $\AHI$ denote the set of non-negative AHI polynomials as in \eqref{eq:family}. Define the \emph{product AHI cone}
\[
\PiAHI \;=\; \operatorname{cone}\Big\{\, \textstyle\prod_{l=1}^{k} p_l \;:\;
k\in\mathbb{N},\ p_1,\dots,p_k \in \AHI \,\Big\},
\]
the conic hull of finite products of AHI polynomials.
\end{definition}

\begin{proposition}\label{prop:pi-sos}
$\PiAHI \subseteq \Sigma_{n,2d}$, and every element of $\PiAHI$ admits an SOS certificate in closed form.
\end{proposition}

\begin{proof}
The SOS cone is closed under multiplication and under non-negative combination: if $p=\sum_i q_i^2$ and $p'=\sum_j r_j^2$, then $pp' = \sum_{i,j}(q_i r_j)^2$. Since every $p_l\in\AHI$ carries the explicit decomposition \eqref{eq:ahi-sonc-decomp} of Theorem~\ref{thm:sos-rep}, the product $\prod_l p_l$ is obtained as a sum of squares by multiplying out those decompositions, with no semidefinite program required. Conic combinations of SOS polynomials are SOS.
\end{proof}

\begin{lemma}\label{lem:circuit-zero-dim}
Let $g\in\mathbb{R}[x_1,\dots,x_n]$ be a non-negative circuit polynomial that is not identically zero, and let $m=\dim\operatorname{New}(g)$ be the dimension of its Newton polytope. Then the real zero set of $g$ in the torus $T=(\mathbb{R}^{*})^{n}$ is contained in a finite union of sets of dimension $n-m$. In particular, if $Z_{0}\subseteq T$ is an irreducible set of dimension $n-1$ and $g$ vanishes identically on $Z_{0}$, then $m=1$: the Newton polytope of $g$ is a segment.
\end{lemma}
 
\begin{proof}
Write $g=\sum_{j=0}^{m} b_j x^{\alpha(j)} + c\,x^{y}$ in the notation of \citep[Section~1]{iliman2016amoebas}, where $\{\alpha(0),\dots,\alpha(m)\}$ affinely spans $\operatorname{New}(g)$ and $y\in\operatorname{New}(g)$. Factoring out the monomial $x^{\alpha(0)}$, which is a unit on $T$ and hence changes no zeros, we may assume $\alpha(0)=0$, so that $\alpha(1),\dots,\alpha(m)$ are linearly independent in $\mathbb{Z}^{n}$ and $y=\sum_{j}\lambda_j\alpha(j)$ is a rational combination of them.
 
Consider the group homomorphism
\[
\pi: T \longrightarrow (\mathbb{R}^{*})^{m},
\qquad
\pi(x) = \big(x^{\alpha(1)},\dots,x^{\alpha(m)}\big).
\]
Because the exponent vectors $\alpha(1),\dots,\alpha(m)$ are linearly independent, $\pi$ is a surjective homomorphism of algebraic groups with kernel of dimension $n-m$, and each fibre $\pi^{-1}(w)$ is a coset of that kernel, hence of dimension $n-m$ (it has finitely many connected components, coming from the sign choices).
 
Every monomial of $g$ is, up to sign, a monomial function of $x^{\alpha(1)},\dots,x^{\alpha(m)}$: this is immediate for the $x^{\alpha(j)}$, and for $x^{y}$ it follows from $y=\sum_j\lambda_j\alpha(j)$ after passing to the common denominator $\mu$ of the $\lambda_j$ and replacing $\pi$ by the corresponding $\mu$-th power map, which does not change fibre dimensions. Consequently $|g|$ is constant on each fibre of $\pi$, and the zero set of $g$ in $T$ is a union of fibres over the zero set $W\subseteq(\mathbb{R}^{*})^{m}$ of the induced function.
 
Now $g$ is non-negative, so by \citep[Theorem~3.8]{iliman2016amoebas} either $|c|<\Theta_g$, in which case $g>0$ on $T$ and $W=\varnothing$, or $|c|=\Theta_g$, in which case by \citep[Proposition~3.4 and Corollary~3.9]{iliman2016amoebas} $g$ attains the value zero exactly at the norm minimizer, so $W$ is finite with at most $2^{m}$ points. In either case the zero set of $g$ in $T$ is a finite union of fibres, each of dimension $n-m$.
 
For the final statement, a set of dimension $n-1$ cannot be contained in a finite union of sets of dimension $n-m$ unless $n-m\ge n-1$, i.e.\ $m\le 1$; and $m=0$ would make $g$ a single monomial square, which does not vanish on $T$. Hence $m=1$.
\end{proof}
 
\begin{theorem}\label{thm:pi-not-sonc}
$\PiAHI \not\subseteq \SONC_{n,2d}$: the product AHI cone is not contained in the SONC cone.
\end{theorem}
 
\begin{proof}
Consider the two AHI polynomials (each the $n=2$ instance of \eqref{eq:sos-identity}; cf.\ the first row of Table~\ref{tab:sos_cases})
\[
p_1(x) = 2\big(x_1^2-x_2^2\big)^2, \qquad p_2(x) = 2\big(x_3^2-x_4^2\big)^2,
\]
and set
\[
H \;=\; p_1 p_2 \;=\; 4\big(x_1^2-x_2^2\big)^2\big(x_3^2-x_4^2\big)^2 \;\in\; \PiAHI .
\]
Suppose, for contradiction, that $H = \sum_{i} g_i$ with each $g_i$ a non-negative circuit polynomial. Work in the torus $T=(\mathbb{R}^{*})^{4}$ and put
\[
Z_{0} \;=\; \big\{x\in T: x_1=x_2\big\},
\qquad
Z_{1} \;=\; \big\{x\in T: x_3=x_4\big\} .
\]
Both are irreducible sets of dimension $3$, and $H$ vanishes identically on each.
 
\emph{Step 1: every summand vanishes on $Z_0\cup Z_1$.}
Let $x_{0}\in Z_{0}\cup Z_{1}$. Then $\sum_i g_i(x_{0}) = H(x_{0}) = 0$ is a sum of non-negative reals, so $g_i(x_{0})=0$ for every $i$. Hence each $g_i$ vanishes identically on both $Z_{0}$ and $Z_{1}$.
 
\emph{Step 2: every summand is a monomial multiple of a binomial square.}
Fix $g_i\not\equiv 0$. It vanishes on the $3$-dimensional irreducible set $Z_{0}$, so by Lemma~\ref{lem:circuit-zero-dim} its Newton polytope is a segment. A circuit polynomial supported on a segment with even vertices is, by \citep[Lemma~3.7 and Theorem~3.8]{iliman2016amoebas}, of the form
\[
g_i \;=\; \mu\,\big(x^{\beta}-\tau\,x^{\gamma}\big)^{2}, \qquad \mu>0,\ \tau>0,
\]
whose zero set in $T$ is the binomial hypersurface $\{x^{\delta}=\tau\}$ with $\delta=\beta-\gamma$.
 
\emph{Step 3: no such $\delta$ exists.}
Suppose $\{x^{\delta}=\tau\}\supseteq Z_{0}$. Substituting $x_2=x_1$ gives
\[
x_1^{\delta_1+\delta_2}\,x_3^{\delta_3}\,x_4^{\delta_4} \;=\; \tau
\qquad\text{for all } x_1,x_3,x_4\in\mathbb{R}^{*},
\]
and a monomial is constant on an open set only if all its exponents vanish. Varying $x_3$ gives $\delta_3=0$, varying $x_4$ gives $\delta_4=0$, and varying $x_1$ gives $\delta_1+\delta_2=0$; evaluating at $x=(1,1,1,1)$ then gives $\tau=1$. Hence $\delta\in\mathbb{Z}(e_1-e_2)$. Applying the identical argument to $Z_{1}$ yields $\delta\in\mathbb{Z}(e_3-e_4)$. Since
\[
\mathbb{Z}(e_1-e_2)\ \cap\ \mathbb{Z}(e_3-e_4) \;=\; \{0\},
\]
we get $\delta=0$ and $\tau=1$, so $g_i = \mu\,x^{2\gamma}\,(1-1)^{2} \equiv 0$, contrary to the choice of $g_i$.
 
Therefore every summand vanishes identically, contradicting $\sum_i g_i = H \not\equiv 0$. Hence $H\notin\SONC_{4,8}$, and since $H\in\PiAHI$ the claim follows.
\end{proof}

\begin{remark}\label{rem:sonc-prior-art}
Theorem~\ref{thm:pi-not-sonc} is a structured instance of the general fact that the SONC cone is not closed under multiplication, established by \citet{dressler2017positivstellensatz} (whose Lemma~4.1 shows in particular that not every square is SONC, so that SONC is neither a preordering nor a quadratic module) and revisited with a simpler construction by \citet{dressler2022optimization}. The contribution here is not the failure of multiplicative closure per se, but that the failure occurs already \emph{within the AM-HM-generated AHI family}, at the lowest possible multiplicative order --- two quartic binomial squares --- and that the resulting cone $\PiAHI$ carries closed-form SOS certificates inherited from its AHI factors (Proposition~\ref{prop:pi-sos}) rather than requiring an SDP for membership.
\end{remark}

The separation is not confined to SONC. Recall that $\SDSOS$, the cone of scaled diagonally dominant sums of squares of \citet{ahmadi2019dsos}, consists of polynomials admitting a Gram matrix $G$ for which some positive diagonal scaling $DGD$ is diagonally dominant; equivalently, $\SDSOS$ is the cone of sums of binomial squares. Every AHI polynomial is a sum of binomial squares by \eqref{eq:ahi-sonc-decomp} and hence lies in $\SDSOS$. Its products need not.
\pagebreak
\begin{proposition}\label{prop:pi-not-sdsos}
$\PiAHI \not\subseteq \SDSOS$.
\end{proposition}

\begin{proof}
Let $H$ be as in Theorem~\ref{thm:pi-not-sonc} and put
\[
u_1=x_1^2x_3^2,\quad u_2=x_1^2x_4^2,\quad u_3=x_2^2x_3^2,\quad u_4=x_2^2x_4^2,
\]
so that $H/4 = (u_1-u_2-u_3+u_4)^2$. The Newton polytope of $H$ forces every Gram basis for $H$ to be contained in $\{u_1,u_2,u_3,u_4\}$, and the only multiplicative relation among these four degree-four monomials is $u_1u_4=u_2u_3$. Writing $H/4 = u^{T}Gu$ with $u=(u_1,u_2,u_3,u_4)^{T}$ and matching coefficients:
\begin{itemize}
\item the monomials $u_i^2$ are distinct, giving $G_{ii}=1$ for $i=1,\dots,4$;
\item the monomials $u_1u_2$, $u_1u_3$, $u_2u_4$, $u_3u_4$ are distinct, giving $G_{12}=G_{13}=G_{24}=G_{34}=-1$;
\item the single relation $u_1u_4=u_2u_3$ gives only the aggregate condition $G_{14}+G_{23}=2$.
\end{itemize}
Positive semidefiniteness forces $|G_{14}|\le\sqrt{G_{11}G_{44}}=1$ and $|G_{23}|\le 1$, so $G_{14}=G_{23}=1$ and the Gram matrix is uniquely $G=vv^{T}$ with $v=(1,-1,-1,1)^{T}$.

Suppose $DGD$ were diagonally dominant for some positive diagonal $D=\operatorname{diag}(d_1,\dots,d_4)$. Then $DGD=ww^{T}$ with $w=Dv$, all $|w_i|>0$, and diagonal dominance of $ww^{T}$ reads $w_i^2\ge\sum_{j\ne i}|w_i w_j|$, i.e.\ $|w_i|\ge\sum_{j\ne i}|w_j|$ for each $i$. Summing these four inequalities gives $\sum_i |w_i| \ge 3\sum_i |w_i|$, which is impossible. Since $G$ is the unique Gram matrix of $H/4$, we conclude $H\notin\SDSOS$.
\end{proof}

\subsection{Relation between \texorpdfstring{$\PiAHI$}{Pi AHI} and SONC}\label{subsec:relation}

\begin{proposition}\label{prop:neither-contains}
Neither of $\PiAHI$ and $\SONC_{n,2d}$ contains the other. However, their intersection is full-dimensional since $\AHI\subseteq \PiAHI\cap\SONC_{n,2d}$.
\end{proposition}

\begin{proof}
That $\PiAHI\not\subseteq\SONC_{n,2d}$ is Theorem~\ref{thm:pi-not-sonc}. Conversely, the Motzkin form $M(x,y,z)=x^4y^2+x^2y^4-3x^2y^2z^2+z^6$ is a non-negative circuit polynomial, hence resides in $\SONC_{3,6}$, but is not SOS \citep{motzkin1967arithmetic}; since $\PiAHI\subseteq\Sigma_{n,2d}$ by Proposition~\ref{prop:pi-sos}, we get $M\notin\PiAHI$ and so $\SONC_{n,2d}\not\subseteq\PiAHI$. Finally, $\AHI\subseteq\PiAHI$ (products of length one) and $\AHI\subseteq\SONC_{n,2d}$ by Proposition~\ref{prop:ahi-in-sonc}, so the intersection contains the full AHI family.
\end{proof}

Summarizing Corollary~\ref{cor:ahi-both} and Propositions~\ref{prop:pi-sos}, \ref{prop:pi-not-sdsos} and \ref{prop:neither-contains}, the picture is
\[
\AHI \;\subseteq\; \SDSOS\cap\SONC_{n,2d}\cap\Sigma_{n,2d},
\qquad
\PiAHI \;\subseteq\; \Sigma_{n,2d},
\qquad
\PiAHI \not\subseteq \SONC_{n,2d},\quad
\PiAHI \not\subseteq \SDSOS .
\]
Thus $\PiAHI$ furnishes a family of polynomials that is explicitly SOS, originates from the AM-HM inequality, and lies outside both of the standard tractable inner approximations of the SOS cone.

\pagebreak
%
%

\definecolor{sosfill}{RGB}{206,203,246}
\definecolor{sosline}{RGB}{83,74,183}
\definecolor{soncfill}{RGB}{159,225,203}
\definecolor{soncline}{RGB}{15,110,86}
\definecolor{prodfill}{RGB}{240,153,123}
\definecolor{prodline}{RGB}{153,60,29}
\definecolor{sdfill}{RGB}{211,209,199}
\definecolor{sdline}{RGB}{95,94,90}
\definecolor{ahifill}{RGB}{250,199,117}
\definecolor{ahiline}{RGB}{133,79,11}
 
\begin{figure}[htbp]
\centering
\begin{tikzpicture}[scale=1.0, every node/.style={font=\small}]
 
  \draw[draw=sdline, fill=sdline, fill opacity=0.06, line width=0.4pt]
        (0,0) ellipse [x radius=7.75cm, y radius=4.375cm];
  \node[font=\small\bfseries] at (0,3.75) {$\mathcal{P}_{n,2d}$ (non-negative)};
 
  \draw[draw=sosline, fill=sosfill, fill opacity=0.45, line width=0.6pt]
        (-1.75,-0.125) ellipse [x radius=5.0cm, y radius=3.375cm];
 
  \draw[draw=soncline, fill=soncfill, fill opacity=0.45, line width=0.6pt]
        (2.5,-0.125) ellipse [x radius=4.375cm, y radius=3.125cm];
 
  \draw[draw=prodline, fill=prodfill, fill opacity=0.55, line width=0.6pt]
        (-1.75,-1.05) ellipse [x radius=4.125cm, y radius=1.15cm];
 
  \draw[draw=sdline, fill=sdfill, fill opacity=0.6, line width=0.6pt]
        (0.8,-1.05) ellipse [x radius=2.05cm, y radius=1.2cm];
 
  \draw[draw=ahiline, fill=ahifill, fill opacity=0.9, line width=0.6pt]
        (0.8,-1.05) ellipse [x radius=1.05cm, y radius=0.65cm];
 
  \node[sosline]  at (-4.9, 1.9) {$\Sigma_{n,2d}$};
  \node[soncline] at ( 5.2, 1.9) {$\mathcal{C}_{n,2d}$};
  \node[sdline, font=\scriptsize] at (1.875,-0.45) {SDSOS};
  \node[ahiline, font=\small\bfseries] at (0.8,-1.05) {AHI};
  \node[prodline, font=\small\bfseries] at (-4.6,-1.05) {$\Pi_{\mathrm{AHI}}$};
 
  \fill (-2.75,-1.7) circle (2pt);
  \draw[dashed, gray, line width=0.3pt] (-2.75,-1.8) -- (-3.6,-4.9);
  \node[anchor=north, align=center, font=\scriptsize] at (-3.6,-4.9)
       {$H=p_1p_2$\\ SOS, not SONC, not SDSOS};
 
  \fill (5.2,0.375) circle (2pt);
  \draw[dashed, gray, line width=0.3pt] (5.2,0.275) -- (3.8,-4.9);
  \node[anchor=north, align=center, font=\scriptsize] at (3.8,-4.9)
       {Motzkin form\\ SONC, not SOS};
 
\end{tikzpicture}
\caption{The AHI family among the standard non-negativity certificates.
Every non-negative AHI polynomial is a sum of binomial squares, hence lies in
$\mathrm{SDSOS}\subseteq\Sigma_{n,2d}\cap\mathcal{C}_{n,2d}$
(Corollary~\ref{cor:ahi-both}), and the containment is strict
(Proposition~\ref{prop:proper}). Closing the family under multiplication gives
$\Pi_{\mathrm{AHI}}$, which remains inside $\Sigma_{n,2d}$
(Proposition~\ref{prop:pi-sos}) but leaves both $\mathcal{C}_{n,2d}$
(Theorem~\ref{thm:pi-not-sonc}) and $\mathrm{SDSOS}$
(Proposition~\ref{prop:pi-not-sdsos}), as witnessed by
$H=4(x_1^2-x_2^2)^2(x_3^2-x_4^2)^2$. The Motzkin form witnesses
$\mathcal{C}_{n,2d}\not\subseteq\Sigma_{n,2d}$
(Proposition~\ref{prop:neither-contains}). Regions are schematic and not to scale.}
\label{fig:cone-hierarchy}
\end{figure}
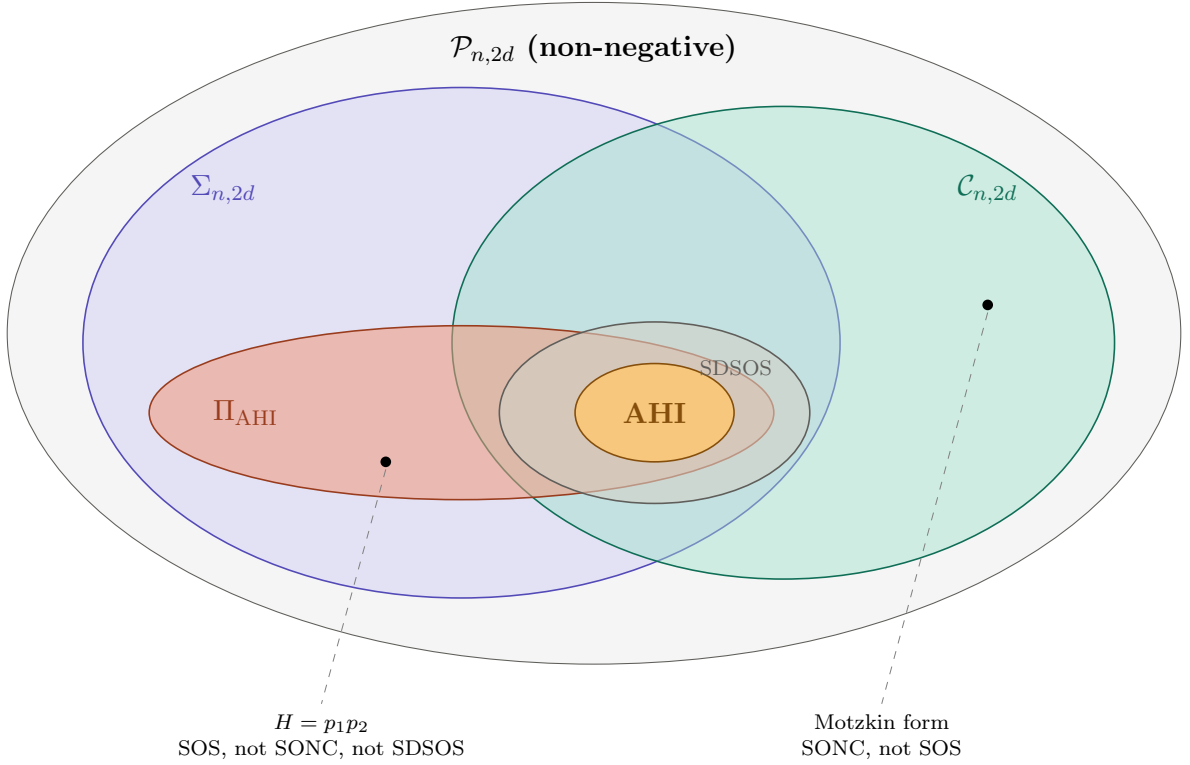
\section{Numerical experiments and optimization}\label{sec:numerics}

In this section we demonstrate the computational advantages of exploiting the AHI structure in polynomial optimization. We first compare standard (dense) SOS relaxation with the AHI-specific sparse relaxation, and then develop a factorized optimization hierarchy for the product cone $\PiAHI$.

\subsection{Experimental setup}\label{subsec:setup}
All numerical experiments were performed on a workstation with an Intel Core i7 processor (2.8 GHz) and 16 GB of RAM, running MATLAB R2023b. The optimization problems were modelled using YALMIP \cite{lofberg2004yalmip} and solved with the semidefinite programming solver SDPT3 \cite{toh1999sdpt3}.

For the \textbf{standard SOS} method we use the full monomial basis of degree up to $d$, giving a Gram matrix of size $\binom{n+d}{d} \times \binom{n+d}{d}$. For the \textbf{AHI method} we exploit the fact that AHI polynomials are composed solely of squared monomials: the Gram matrix is indexed by the half-support \eqref{eq:half-support}, of size $n(n-1)+1$, which is dramatically smaller than the full basis.

\subsection{Unconstrained minimization and scalability}\label{subsec:scalability}
We first consider unconstrained minimization of randomly generated AHI polynomials. To analyse scalability we vary the dimension $n$ and the degree $2d$, recording the number of decision variables in the resulting SDP ($N_{\mathrm{var}}$), the size of the positive semidefinite cone ($K_{\mathrm{size}}$, the dimension of the Gram matrix, equal to the number of basis monomials in the SOS decomposition), and the CPU time of the solver in seconds.

Table~\ref{tab:scalability} summarizes the results. The ``Speedup'' column is the ratio of dense solver time to AHI solver time. The symbol ``OOM'' indicates that the standard solver exceeded available memory and failed to construct the problem.

\begin{table}[h!]
\centering
\resizebox{\textwidth}{!}{%
\begin{tabular}{@{}cc|rrr|rrr|r@{}}
\toprule
\multicolumn{2}{c|}{\textbf{Problem}} & \multicolumn{3}{c|}{\textbf{Standard SOS (dense)}} & \multicolumn{3}{c|}{\textbf{AHI method (sparse)}} & \multicolumn{1}{c}{\textbf{Comparison}} \\ \midrule
$n$ & $2d$ & $N_{\mathrm{var}}$ & $K_{\mathrm{size}}$ & Time (s) & $N_{\mathrm{var}}$ & $K_{\mathrm{size}}$ & Time (s) & \textbf{Speedup} \\ \midrule
4 & 16 & 122{,}760 & 495 & 186.12 & 91 & 13 & 0.59 & \textbf{315x} \\
4 & 24 & -- & -- & OOM & 91 & 13 & 0.25 & $\infty$ \\
6 & 24 & -- & -- & OOM & 496 & 31 & 0.54 & $\infty$ \\
7 & 28 & -- & -- & OOM & 946 & 43 & 1.21 & $\infty$ \\
8 & 32 & -- & -- & OOM & 1{,}653 & 57 & 2.79 & $\infty$ \\
5 & 30 & -- & -- & OOM & 231 & 21 & 0.29 & $\infty$ \\
6 & 36 & -- & -- & OOM & 496 & 31 & 0.51 & $\infty$ \\
4 & 40 & -- & -- & OOM & 91 & 13 & 0.20 & $\infty$ \\
10 & 40 & -- & -- & OOM & 4{,}186 & 91 & 17.21 & $\infty$ \\ \bottomrule
\end{tabular}%
}
\caption{Comparison of standard SOS and AHI relaxations. Note that $K_{\mathrm{size}}$ in the AHI columns realizes the bound $n(n-1)+1$ of \eqref{eq:half-support}.}
\label{tab:scalability}
\end{table}

As shown in Table~\ref{tab:scalability}, the standard dense SOS relaxation becomes intractable almost immediately, failing with OOM for nearly all instances, whereas the AHI method solves every instance efficiently. Even for the smallest case $(n=4,\,2d=16)$ the standard method requires more than three minutes while the AHI method finishes in under a second. For all subsequent cases, including high-degree instances up to $2d=40$, the dense relaxation fails completely due to memory constraints while the AHI method remains stable and fast.

For the one solvable case $(n=4,\,2d=16)$, both the dense and the AHI sparse relaxation converged to the same optimal lower bound (to solver tolerance $10^{-8}$), confirming that the AHI method achieves its speedup without loss of precision or optimality.

\subsection{Visualizing the complexity gap}\label{subsec:gap}
Figure~\ref{fig:time_plot} plots computation time against problem instances sorted by increasing complexity, on a logarithmic vertical axis to accommodate the disparity in performance. The red curve represents the standard dense relaxation; its steep trajectory corresponds to the combinatorial explosion of the monomial basis size $\binom{n+d}{d}$, which quickly exhausts available memory. In contrast the AHI method (blue curve) exhibits gentle and stable growth: even for the largest instance considered $(n=10,\,2d=40)$ the computation time remains under 20 seconds. The AHI method therefore does not merely offer a constant-factor speedup but alters the effective complexity of the problem for this family.

\begin{figure}[htbp]
\centering
\includegraphics[width=0.85\textwidth]{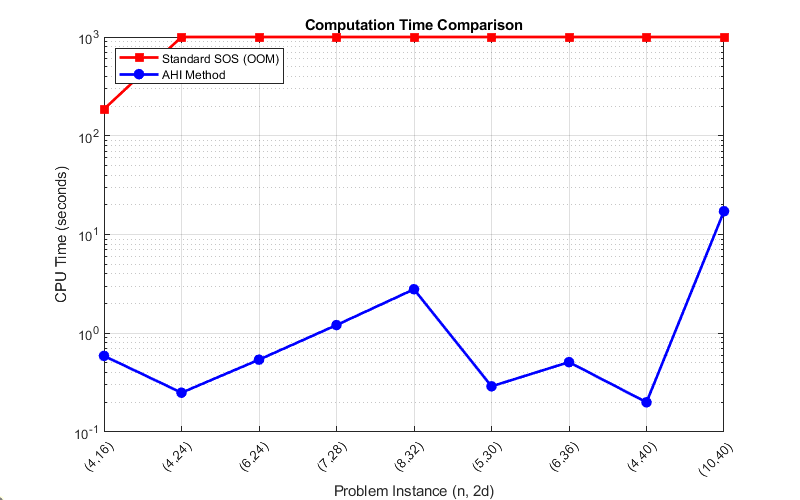}
\caption{Log-scale comparison of CPU time. The red line indicates the threshold where the standard method fails (out of memory); the blue line shows the actual performance of the AHI method.}
\label{fig:time_plot}
\end{figure}

\subsection{Constrained optimization examples}\label{subsec:constrained}

We now apply the method to constrained problems whose objective or constraints belong to the AHI family.

\begin{example}[Non-convex objective over a Euclidean ball]\label{ex:num1}
\[
\begin{aligned}
& \min_{x\in \mathbb{R}^{2}} \quad 3x_1^4 + 6x_1^4 x_2^2 + 3x_2^4 + 2x_1^2 x_2^4 + 6x_2^2 + 2x_1^2 - 18x_1^2 x_2^2\\
& \quad\text{s.t.}\quad  3 - x_1^2 - x_2^2 \geq 0 .
\end{aligned}
\]
The objective is a non-convex AHI polynomial. Solving the relaxation at order $k=2$ yields the global minimum $f^* = 0$ with optimizer $x^* \approx (0.6544,\,-0.0805)$.
\end{example}

\begin{example}\label{ex:num2}
\[
\begin{aligned}
   \min_{x\in\mathbb{R}^2} \quad& \big(x_1^2x_2^4 + 2x_1^4x_2^2\big)\big(2x_1^2x_2^4 + x_1^4x_2^2\big) - x_1^6x_2^6\\
   \text{s.t.}\quad &  - x_1^2 + x_2 \geq 0, \\
  & -x_1 - x_2 +(x_2 - x_1 + 0.65)^2 + 3.85 \geq 0,\\
  & 0\leq x_i\leq 5,\quad i=1,2 .
\end{aligned}
\]
The objective is a non-convex AHI polynomial and the feasible region is defined by AHI constraints. The Lasserre hierarchy converges at second order to $f^* \approx -0.0012$ at $x^*=(0.1481,\,0.1224)$.
\end{example}

\begin{example}[Higher-dimensional box constraints]\label{ex:num3}
\[
\begin{aligned}
& \min_{x\in \mathbb{R}^{4}}  \quad  \big(x_1^2 + x_2^2 + x_3^2 + x_4^2\big)\big(x_2^2x_3^2x_4^2 + x_1^2x_3^2x_4^2 + x_1^2x_2^2x_4^2 + x_1^2x_2^2x_3^2\big) - 7x_1^2x_2^2x_3^2x_4^2\\
& \quad\text{s.t.}\quad 1\leq x_{i}\leq 3,\quad i=1,2,3,4 .
\end{aligned}
\]
The sparse AHI formulation solves this in 0.8 seconds, yielding optimal value $9$ at $x^* = (1,1,1,1)$. The dense relaxation requires substantially more memory and time to process the fourth-order moment matrix in four variables.
\end{example}

\subsection{A factorized optimization hierarchy for \texorpdfstring{$\PiAHI$}{Pi AHI}}\label{subsec:factorized-hierarchy}

Let $f = \prod_{l=1}^{k} f_l$ be a product of AHI polynomials, i.e.\ a generator of $\PiAHI$ (Definition~\ref{def:piahi}), and let $K \subseteq \mathbb{R}^{n}$ be a basic closed semialgebraic set. Consider
\begin{equation}\label{eq:factorized-obj}
f^{*}_{K} \;=\; \inf_{x \in K}\ \prod_{l=1}^{k} f_l(x).
\end{equation}
Each factor $f_l$ carries the explicit sum-of-binomial-squares certificate \eqref{eq:ahi-sonc-decomp} and the small half-support \eqref{eq:half-support}. The purpose of this subsection is to convert that per-factor structure into a method for the \emph{product} \eqref{eq:factorized-obj} that avoids forming the dense relaxation of degree $\sum_l \deg f_l$.

\subsubsection{Level A: closed-form factored certificate}\label{subsubsec:levelA}

Multiplying the decompositions of the factors yields an SOS representation of $f$ with no semidefinite program at all. Concretely, if $f_l=\sum_{a}\sigma_{l,a}^2$, then
\begin{equation}\label{eq:product-of-squares}
f \;=\; \prod_{l=1}^k \Big(\sum_a \sigma_{l,a}^2\Big) \;=\; \sum_{a_1,\dots,a_k}\Big(\prod_{l=1}^k \sigma_{l,a_l}\Big)^{2},
\end{equation}
a sum of $\prod_l\big(\binom{n_l}{2}+1\big)$ explicit squares. Writing $B_l$ for the half-support of $f_l$, a Gram basis for $f$ is contained in the Minkowski sum $B = B_1 + \cdots + B_k$, so by \eqref{eq:half-support}
\begin{equation}\label{eq:basis-size}
|B| \;\le\; \prod_{l=1}^{k}\big(n_l(n_l-1)+1\big),
\end{equation}
in contrast with the dense half-degree basis of size $\binom{n+D}{D}$, where $D = \tfrac12\sum_l \deg f_l$.

\begin{example}\label{ex:levelA}
For $f_1 = 2(x_1^2-x_2^2)^2$ and $f_2 = 2(x_3^2-x_4^2)^2$ we have $B_1=\{x_1^2,x_2^2\}$ and $B_2=\{x_3^2,x_4^2\}$, each of size two, and
\[
f_1f_2 = \Big(2\big(x_1^2x_3^2 - x_1^2x_4^2 - x_2^2x_3^2 + x_2^2x_4^2\big)\Big)^{2},
\]
a single square on the basis $B_1+B_2=\{x_1^2x_3^2,\,x_1^2x_4^2,\,x_2^2x_3^2,\,x_2^2x_4^2\}$ of size $4$, against a dense degree-eight basis in four variables of size $\binom{8}{4}=70$.
\end{example}

\subsubsection{Level B: factorized lower bound}\label{subsubsec:levelB}

If every factor is non-negative on $K$ --- decidable in closed form by Theorems~\ref{thm:sufficient} and~\ref{thm:necessary} --- then from $f_l(x) \ge \inf_{K} f_l \ge 0$ for all $l$ we obtain
\begin{equation}\label{eq:factorized-bound}
\inf_{x\in K}\ \prod_{l=1}^{k} f_l(x)
\;\ge\;
\prod_{l=1}^{k}\Big(\inf_{x\in K} f_l(x)\Big).
\end{equation}
The right-hand side decouples \eqref{eq:factorized-obj} into $k$ independent AHI subproblems, each solvable by the sparse relaxation of Section~\ref{subsec:scalability} or by geometric programming \citep{iliman2016lower}, at cost governed by the small per-factor size $|B_l|$ rather than by $|B|$ or $\binom{n+D}{D}$.

\begin{proposition}\label{prop:factorized-exact}
Suppose $f_1,\dots,f_k$ are non-negative on $K$ and attain their minima over $K$ at a common point $x^{*}\in K$, i.e.\ $f_l(x^{*}) = \inf_K f_l$ for every $l$. Then \eqref{eq:factorized-bound} holds with equality, and $f^{*}_{K} = \prod_l (\inf_K f_l)$ is certified by solving the $k$ subproblems alone.
\end{proposition}

\begin{proof}
By \eqref{eq:factorized-bound}, $f^{*}_K \ge \prod_l(\inf_K f_l)$. Conversely, evaluating at the common minimizer, $f(x^{*}) = \prod_l f_l(x^{*}) = \prod_l (\inf_K f_l) \ge f^{*}_K$. The two inequalities coincide.
\end{proof}

\begin{remark}\label{rem:common-min}
The common-minimizer hypothesis holds automatically for AHI factors on a positive box $K=[a,b]^{n}$ with $a>0$: every monomial of a non-negative AHI polynomial has non-negative exponents, so on the positive orthant each factor is minimized at the corner $(a,\dots,a)$ simultaneously. When the factors are minimized at different points --- for instance when the weights $\lambda_i$ pull an optimizer off the diagonal, or when $K$ is not a positive box --- the bound \eqref{eq:factorized-bound} is in general strict and one passes to Level~C.
\end{remark}

\subsubsection{Level C: multiplier hierarchy for overlapping factors}\label{subsubsec:levelC}

When the factors share variables and no common minimizer exists, \eqref{eq:factorized-bound} may be loose. In that case we retain the factors as building blocks but allow them to be weighted by polynomial multipliers. It is convenient to work with the \emph{shifted} factors
\begin{equation}\label{eq:shifted-factors}
h_l \;=\; f_l - \inf_{K} f_l \;\ge\; 0 \ \text{ on } K,
\qquad \inf_{K} h_l = 0,
\end{equation}
whose minima are already available from the $k$ cheap subproblems of Level~B. We then seek the largest $\gamma$ admitting a certificate
\begin{equation}\label{eq:multiplier-hierarchy}
f - \gamma
\;=\;
\sum_{S \subseteq \{1,\dots,k\}} \sigma_{S}\prod_{l\in S} h_l,
\qquad \sigma_{S}\ \text{SOS with } \deg \sigma_S \le d,
\end{equation}
where the empty product equals $1$. Denote by $\gamma_d$ the optimal value. Since the $h_l$ are pre-certified non-negative on $K$, every feasible point of \eqref{eq:multiplier-hierarchy} yields a valid lower bound $\gamma_d \le f^{*}_K$, and increasing $d$ enlarges the feasible set, so $\gamma_0 \le \gamma_2 \le \gamma_4 \le \cdots \le f^{*}_{K}$.

\begin{proposition}\label{prop:level-c-base}
At $d=0$ (constant multipliers) the certificate \eqref{eq:multiplier-hierarchy} recovers the factorized bound \eqref{eq:factorized-bound}, i.e.\ $\gamma_0 = \prod_{l}(\inf_K f_l)$.
\end{proposition}

\begin{proof}
Write $m_l = \inf_K f_l$, so $f_l = h_l + m_l$. Expanding,
\[
f \;=\; \prod_{l} (h_l + m_l) \;=\; \prod_l m_l \;+\!\!\sum_{\varnothing\neq S\subseteq\{1,\dots,k\}}\!\!\Big(\prod_{l\notin S} m_l\Big)\prod_{l\in S} h_l ,
\]
which is exactly \eqref{eq:multiplier-hierarchy} with $\gamma=\prod_l m_l$, $\sigma_{\varnothing}=0$ and the non-negative constants $\sigma_S = \prod_{l\notin S} m_l$ for $S\neq\varnothing$. Hence $\gamma_0 \ge \prod_l m_l$. Conversely, for constant multipliers the right-hand side of \eqref{eq:multiplier-hierarchy} is a non-negative combination of the $h_l$-products, and evaluating at any point where all $h_l$ vanish simultaneously (or taking an infimizing sequence) gives $\gamma\le\prod_l m_l$.
\end{proof}

\begin{remark}\label{rem:shift-needed}
Using the shifted factors \eqref{eq:shifted-factors} rather than the raw $f_l$ is essential for Proposition~\ref{prop:level-c-base}. With the unshifted factors, matching the top-degree terms forces $\sigma_{\{1,\dots,k\}}=1$ and then $\sigma_{\varnothing} = -\gamma$, which is SOS only for $\gamma\le 0$; the resulting $d=0$ bound would be worse than \eqref{eq:factorized-bound}. With the shift, $\gamma_d$ interpolates between the free factorized bound at $d=0$ and a dense relaxation for large $d$, while the products $\prod_{l \in S} h_l$ need never be expanded into the dense monomial basis.
\end{remark}


A detailed analysis of \eqref{eq:multiplier-hierarchy} for AHI factors --- in particular conditions under which $\gamma_d \to f^{*}_K$, the degree at which finite convergence occurs, and the relation to the constrained SONC and geometric programming framework of \citet{dressler2019approach} --- is left for future work. We stress that convergence does not follow from the construction alone and would require an Archimedean-type condition on the module generated by the $h_l$.

\subsubsection{Comparison with generic sparsity}\label{subsubsec:sparsity}

The decoupling in \eqref{eq:factorized-bound} is \emph{not} recovered by correlative sparsity \citep{waki2006sums}: expanding a product $\prod_l f_l$ into monomials couples variables across factors, since a monomial of $f$ generically contains variables from several $f_l$, so the correlative sparsity graph becomes a single clique and no decomposition is detected. The sparsity of $\PiAHI$ is \emph{multiplicative} rather than additive. Term sparsity \citep{wang2021tssos} does apply --- every generator of $\PiAHI$ has only even exponents, hence the maximal sign-symmetry group $(\mathbb{Z}/2)^{n}$, which TSSOS exploits --- but term sparsity operates on the expanded polynomial and likewise does not exploit the factor structure. The factorized hierarchy above is therefore complementary to, not subsumed by, existing sparsity techniques.

\subsubsection{Worked example}\label{subsubsec:worked}

Let
\[
\begin{aligned}
f_1 &= 2x_2^2x_3^2x_4^2(x_1^4+x_5^4) + 2x_1^2x_3^2x_4^2(x_2^4+x_5^4)
     + 2x_1^2x_2^2x_4^2(x_3^4+x_5^4) + 2x_1^2x_2^2x_3^2(x_4^4+x_5^4)\\
    &\quad + x_3^2x_4^2x_5^2(x_1^4+x_2^4) + x_2^2x_4^2x_5^2(x_1^4+x_3^4)
     + x_2^2x_3^2x_5^2(x_1^4+x_4^4) + x_1^2x_4^2x_5^2(x_2^4+x_3^4)\\
    &\quad + x_1^2x_3^2x_5^2(x_2^4+x_4^4) + x_1^2x_2^2x_5^2(x_3^4+x_4^4)
     - 4x_1^2x_2^2x_3^2x_4^2x_5^2,
\end{aligned}
\]
\[
\begin{aligned}
f_2 &= x_3^2x_4^2(x_1^4+x_2^4) + x_2^2x_4^2(x_1^4+x_3^4)
     + x_2^2x_3^2(x_1^4+x_4^4) + x_1^2x_4^2(x_2^4+x_3^4)\\
    &\quad + x_1^2x_3^2(x_2^4+x_4^4) + x_1^2x_2^2(x_3^4+x_4^4)
     - 3x_1^2x_2^2x_3^2x_4^2.
\end{aligned}
\]
Both are AHI polynomials in the sense of \eqref{eq:family} with $\alpha(i)=2e_i$: for $f_2$ we read off $\lambda_1=\dots=\lambda_4=1$ and $c=7$, so $\sum_i\lambda_i^2-c = 4-7 = -3$; for $f_1$ we read off $\lambda_1=\dots=\lambda_4=1$, $\lambda_5=2$ and $c=12$, so $\sum_i\lambda_i^2-c = 8-12 = -4$. Both satisfy the threshold of Theorem~\ref{thm:sufficient} (namely $-3\ge -12$ and $-4\ge -36$), so both factors are non-negative and
\[
f = f_1 f_2 \in \PiAHI
\]
is a degree-$18$ form in five variables whose factors \emph{share} the four variables $x_1,\dots,x_4$.

Minimize $f$ over the box $K = [1,3]^5$. By Remark~\ref{rem:common-min} each factor is minimized at the common corner $x^{*}=(1,1,1,1,1)$, with
\[
\inf_K f_1 = f_1(x^{*}) = 24, \qquad \inf_K f_2 = f_2(x^{*}) = 9 .
\]
By Proposition~\ref{prop:factorized-exact} the factorized bound is tight, so
\[
f^{*}_{K} \;=\; (\inf_K f_1)(\inf_K f_2) \;=\; 24 \cdot 9 \;=\; 216 ,
\]
certified by two AHI subproblems of Gram size $|B_1| = 5\cdot4+1 = 21$ and $|B_2| = 4\cdot3+1 = 13$. Note that the overlap of variables does not obstruct exactness here, because the AHI symmetry aligns both minimizers at the same corner. A dense relaxation of the same problem would require the degree-$9$ monomial basis in five variables, of size $\binom{14}{5} = 2002$, against $\prod_l|B_l| = 21\cdot 13 = 273$ for the factored basis \eqref{eq:basis-size}.

\subsection{Summary of the numerical results}\label{subsec:summary}
The experiments confirm that the algebraic structure of the AHI family admits highly efficient SDP representations. Reducing the Gram matrix from a combinatorial function of the degree to the linear-in-$n^2$ size \eqref{eq:half-support} allows problems up to degree 40 and dimension 10 to be solved on standard hardware, and the method integrates into the Lasserre hierarchy for constrained optimization, often yielding global minima at low relaxation orders ($k=2$). For products of AHI polynomials, the factorized hierarchy of Section~\ref{subsec:factorized-hierarchy} replaces a single large relaxation by several small ones, with exactness guaranteed whenever the factors share a minimizer.

\section{Conclusions}\label{sec:conclusions}

In this paper we introduced a new family of non-negative polynomials, the AHI polynomials, constructed with the help of the arithmetic mean--harmonic mean inequality. We investigated the algebraic and geometric properties of this family, focusing on its relationship with sum-of-squares polynomials and on convexity.

A key theoretical contribution is the proof that, for the AHI family, the set of non-negative polynomials coincides exactly with the set of SOS polynomials (Theorems~\ref{thm:sufficient} and~\ref{thm:necessary}). This is significant, as the property generally fails for arbitrary polynomials of degree at least four in more than two variables, and it constitutes an addendum to Hilbert's 17th problem.

We then examined convexity, showing that AHI polynomials are in general non-convex (Theorems~\ref{thm:nonconvex-2var} and~\ref{thm:nonconvex-3var}) but that for specific monomial structures convexity implies SOS-convexity (Theorem~\ref{thm:sos-convex}), providing a tractable algebraic certificate in those cases.

Section~\ref{sec:cones} located the family among the standard certificates: every AHI polynomial is both SOS and SONC, and the containment in the intersection is strict, the obstruction being the rank-one rigidity \eqref{eq:coeff-rigidity} of the coefficient matrix. Closing the family under multiplication yields the cone $\PiAHI$, which is SOS with closed-form certificates but escapes both SONC (Theorem~\ref{thm:pi-not-sonc}) and SDSOS (Proposition~\ref{prop:pi-not-sdsos}). This gives an explicit, AM-HM-generated family of certificates lying outside the two standard tractable inner approximations of the SOS cone.

Finally, we leveraged the sparse structure of AHI polynomials in optimization. Numerical experiments confirm substantial computational advantages within the Lasserre hierarchy, and the factorized hierarchy of Section~\ref{subsec:factorized-hierarchy} exploits the multiplicative structure of $\PiAHI$ in a way that correlative and term sparsity do not. Future work may address the convergence of the multiplier hierarchy \eqref{eq:multiplier-hierarchy}, a characterization of $\PiAHI\cap\SONC_{n,2d}$, and extensions of the construction to other classical inequalities.

\bibliographystyle{apalike}
\bibliography{reference}

\pagebreak
\appendix

\section{Explicit SOS decompositions of AHI polynomials}\label{app:sos-table}

\begin{table}[h!]
    \centering
    \renewcommand{\arraystretch}{2.2}
    \begin{tabular}{|p{7cm}|p{7cm}|}
    \hline
     \textbf{AHI polynomial}  & \textbf{SOS representation} \\
    \hline

    $2(x_1^4 + x_2^4) - 4x_1^2 x_2^2$
    &
    $2(x_1^2 - x_2^2)^2$
    \\ \hline

    $(x_1^4+x_2^4)x_3^2 + 2(x_1^4+x_3^4)x_2^2 + 2(x_2^4+x_3^4)x_1^2 - 10x_1^2 x_2^2 x_3^2$
    &
    $(x_1^2-x_2^2)^2 x_3^2 + 2(x_1^2-x_3^2)^2 x_2^2 + 2(x_2^2-x_3^2)^2 x_1^2$
    \\ \hline

    $
    \begin{aligned}
    & (x_1^4x_2^4 + x_2^4x_3^4)x_3^2
     + 2(x_1^4x_2^4 + x_3^4)x_2^2x_3^2 \\
    & + 0.5(x_2^4x_3^4 + x_3^4)x_1^2x_2^2
     - 3.75\, x_1^2 x_2^4 x_3^4
    \end{aligned}
    $
    &
    $
    \begin{aligned}
    & (x_1^2x_2^2 - x_2^2x_3^2)^2 x_3^2
     + 2(x_1^2x_2^2 - x_3^2)^2 x_2^2x_3^2 \\
    & + 0.5(x_2^2x_3^2 - x_3^2)^2 x_1^2x_2^2
     + 3.25\, x_1^2 x_2^4 x_3^4
    \end{aligned}
    $
    \\ \hline

    $6(x_1^4 x_3^4 + x_2^8 x_4^4) - 2 x_1^2 x_2^4 x_3^2 x_4^2$
    &
    $6(x_1^2 x_3^2 - x_2^4 x_4^2)^2 + 10\, x_1^2 x_2^4 x_3^2 x_4^2$
    \\ \hline

    \begin{minipage}{7cm}
\vspace{5pt}
$x_3^2 x_4^2 (x_1^4 + x_2^4) + x_2^2 x_4^2 (x_1^4 + x_3^4)$ \\
$+ x_2^2 x_3^2 (x_1^4 + x_4^4) + x_1^2 x_4^2 (x_2^4 + x_3^4)$ \\
$+ x_1^2 x_3^2 (x_2^4 + x_4^4) + x_1^2 x_2^2 (x_3^4 + x_4^4)$ \\
$- 3 x_1^2 x_2^2 x_3^2 x_4^2$
\vspace{5pt}
\end{minipage} &
\begin{minipage}{7cm}
\vspace{5pt}
$x_3^2 x_4^2 (x_1^2 - x_2^2)^2 +
x_2^2 x_4^2 (x_1^2 - x_3^2)^2 +$ \\
$ x_2^2 x_3^2 (x_1^2 - x_4^2)^2 +
x_1^2 x_4^2 (x_2^2 - x_3^2)^2 +$ \\
$ x_1^2 x_3^2 (x_2^2 - x_4^2)^2 +
x_1^2 x_2^2 (x_3^2 - x_4^2)^2$ \\
$+\,9 x_1^2 x_2^2 x_3^2 x_4^2$
\vspace{5pt}
\end{minipage} \\ \hline

    \begin{minipage}{7cm}
\vspace{5pt}
$2 x_2^2 x_3^2 x_4^2 (x_1^4 + x_5^4) + 2 x_1^2 x_3^2 x_4^2 (x_2^4 + x_5^4)$ \\
$+ 2 x_1^2 x_2^2 x_4^2 (x_3^4 + x_5^4) + 2 x_1^2 x_2^2 x_3^2 (x_4^4 + x_5^4)$ \\
$+ x_3^2 x_4^2 x_5^2 (x_1^4 + x_2^4) + x_2^2 x_4^2 x_5^2 (x_1^4 + x_3^4)$ \\
$+ x_2^2 x_3^2 x_5^2 (x_1^4 + x_4^4) + x_1^2 x_4^2 x_5^2 (x_2^4 + x_3^4)$ \\
$+ x_1^2 x_3^2 x_5^2 (x_2^4 + x_4^4) + x_1^2 x_2^2 x_5^2 (x_3^4 + x_4^4)$ \\
$- 4 x_1^2 x_2^2 x_3^2 x_4^2 x_5^2$
\vspace{5pt}
\end{minipage} &
\begin{minipage}{7cm}
\vspace{5pt}
$2 x_2^2 x_3^2 x_4^2 (x_1^2 - x_5^2)^2 +
2 x_1^2 x_3^2 x_4^2 (x_2^2 - x_5^2)^2 +$ \\
$ 2 x_1^2 x_2^2 x_4^2 (x_3^2 - x_5^2)^2 +
2 x_1^2 x_2^2 x_3^2 (x_4^2 - x_5^2)^2 +$ \\
$ x_3^2 x_4^2 x_5^2 (x_1^2 - x_2^2)^2 +
x_2^2 x_4^2 x_5^2 (x_1^2 - x_3^2)^2 +$ \\
$ x_2^2 x_3^2 x_5^2 (x_1^2 - x_4^2)^2 +
x_1^2 x_4^2 x_5^2 (x_2^2 - x_3^2)^2 +$ \\
$ x_1^2 x_3^2 x_5^2 (x_2^2 - x_4^2)^2 +
x_1^2 x_2^2 x_5^2 (x_3^2 - x_4^2)^2$ \\
$+\,24 x_1^2 x_2^2 x_3^2 x_4^2 x_5^2$
\vspace{5pt}
\end{minipage} \\ \hline

     \end{tabular}
      \caption{SOS decompositions of AHI polynomials. The last two rows are the factors $f_2$ and $f_1$ used in the worked example of Section~\ref{subsubsec:worked}.}
    \label{tab:sos_cases}
\end{table}

\clearpage

\section{Hessian computation for Theorem \texorpdfstring{\ref{thm:nonconvex-2var}}{4.1}}\label{app:hessian}

\begin{proof}[Addendum to the proof of Theorem~\ref{thm:nonconvex-2var}]
Consider the structure
\begin{equation}\label{eq:nonconvex-form}
f(x)=\lambda_1\lambda_2\big(x_1^{\alpha_1}x_2^{\beta_1} + x_1^{\alpha_2}x_2^{\beta_2}\big) + (\lambda_1^2 + \lambda_2^2 - c)\,x_1^{\alpha_1 + \alpha_2}x_2^{\beta_1 + \beta_2},
\end{equation}
where $\lambda_1,\lambda_2\geq 0$ and $\alpha_i,\beta_i\geq2$ are even. If $f$ is convex then all principal minors of the Hessian $H(x)$ are non-negative.

Consider $H_{11}$. If $\lambda_1^2 + \lambda_2^2 - c$ is positive then all terms are positive and $H_{11}$ is SOS. If $\lambda_1^2 + \lambda_2^2 - c$ is negative, then $H_{11}$ is non-negative if and only if it is SOS, which requires
\[
\begin{aligned}
& (\lambda_1^2 + \lambda_2^2 - c)(\alpha_1 + \alpha_2)(\alpha_1 + \alpha_2 - 1)\\
= \quad& 2\sqrt{\lambda_1\lambda_2(2\alpha_1)(2\alpha_1 - 1)}\sqrt{\lambda_1\lambda_2(2\alpha_2)(2\alpha_2 - 1)}\\
= \quad& 2\lambda_1\lambda_2\sqrt{(2\alpha_1)(2\alpha_1 - 1)(2\alpha_2)(2\alpha_2 - 1)} ,
\end{aligned}
\]
a contradiction, since the left-hand side is negative while the right-hand side is positive. The same holds for $H_{22}$. Hence the first necessary condition for convexity of $f$ is
\[
\lambda_1^2 + \lambda_2^2 - c > 0 .
\]

We now compute $H_{11}H_{22} - H_{12}^2$:
\[
\begin{aligned}
   & H_{11}H_{22} - H_{12}^2  \\& =\lambda_1^2\lambda_2^2 (4\alpha_1\beta_1)(2\alpha_1 - 1)(2\beta_1 - 1)x_1^{4\alpha_1 - 2}x_2^{4\beta_1 - 2} \\
    & + \lambda_1^2\lambda_2^2 (4\alpha_2\beta_1)(2\alpha_2 - 1)(2\beta_1 - 1)x_1^{2\alpha_1 + 2\alpha_2 - 2}x_2^{2\beta_1 + 2\beta_2 - 2} \\
    & + \lambda_1\lambda_2(\lambda_1^2 + \lambda_2^2 - c)(\alpha_1 + \alpha_2)(\alpha_1 + \alpha_2 -1)(2\beta_1)(2\beta_1 - 1)x_1^{3\alpha_1 + \alpha_2 - 2}x_2^{3\beta_1 + \beta_2 - 2} \\
    & + \lambda_1^2\lambda_2^2 (4\alpha_1\beta_2)(2\alpha_1 - 1)(2\beta_2 - 1)x_1^{2\alpha_1 + 2\alpha_2 - 2}x_2^{2\beta_1 + 2\beta_2 - 2} \\
    & + \lambda_1^2\lambda_2^2 (4\alpha_2\beta_2)(2\alpha_2 - 1)(2\beta_2 - 1)x_1^{4\alpha_2 - 2}x_2^{4\beta_2 - 2} \\
    & + \lambda_1\lambda_2(\lambda_1^2 + \lambda_2^2 - c)(\alpha_1 + \alpha_2)(\alpha_1 + \alpha_2 -1)(2\beta_2)(2\beta_2 - 1)x_1^{\alpha_1 + 3\alpha_2 - 2}x_2^{\beta_1 + 3\beta_2 - 2} \\
    & + \lambda_1\lambda_2(\lambda_1^2 + \lambda_2^2 - c)(2\alpha_1)(2\alpha_1 - 1)(\beta_1 + \beta_2)(\beta_1 + \beta_2 - 1)x_1^{3\alpha_1 + \alpha_2 -2}x_2^{3\beta_1 + \beta_2 -2}\\
    & + \lambda_1\lambda_2(\lambda_1^2 + \lambda_2^2 - c)(2\alpha_2)(2\alpha_2 - 1)(\beta_1 + \beta_2)(\beta_1 + \beta_2 - 1)x_1^{\alpha_1 + 3\alpha_2 -2}x_2^{\beta_1 + 3\beta_2 -2}\\
    & + (\lambda_1^2 + \lambda_2^2 - c)^2 (\alpha_1 + \alpha_2)(\alpha_1 + \alpha_2 -1)(\beta_1 + \beta_2)(\beta_1 + \beta_2 - 1)x_1^{2\alpha_1 + 2\alpha_2 -2}x_2^{2\beta_1 + 2\beta_2 -2}\\
    & -\lambda_1^2\lambda_2^2 (16\alpha_1^2\beta_1^2)x_1^{4\alpha_1 -2}x_2^{4\beta_1 -2}  -\lambda_1^2\lambda_2^2 (16\alpha_2^2\beta_2^2)x_1^{4\alpha_2 -2}x_2^{4\beta_2 -2}\\
    & -2\lambda_1^2\lambda_2^2 (4\alpha_1\beta_1)(4\alpha_2\beta_2) x_1^{2\alpha_1 + 2\alpha_2 - 2}x_2^{2\beta_1 + 2\beta_2 - 2}\\
    & - 2\lambda_1\lambda_2(\lambda_1^2 + \lambda_2^2 - c)(4\alpha_1\beta_1)(\alpha_1 +\alpha_2)(\beta_1 + \beta_2)x_1^{3\alpha_1 + \alpha_2 -2}x_2^{3\beta_1 + \beta_2 - 2} \\
    & - 2\lambda_1\lambda_2(\lambda_1^2 + \lambda_2^2 - c)(4\alpha_2\beta_2)(\alpha_1 +\alpha_2)(\beta_1 + \beta_2)x_1^{\alpha_1 + 3\alpha_2 -2}x_2^{\beta_1 + 3\beta_2 - 2} \\
    & - (\lambda_1^2 + \lambda_2^2 - c)^2 (\alpha_1 +\alpha_2)^2 (\beta_1 + \beta_2)^2 x_1^{2\alpha_1 + 2\alpha_2 -2}x_2^{2\beta_1 + 2\beta_2 - 2}.
\end{aligned}
\]
We now collect the coefficients of equal monomials.

\medskip
\noindent\emph{Coefficient of $x_1^{4\alpha_1 -2}x_2^{4\beta_1 - 2}$:}
\[
\begin{aligned}
  \,&  \lambda_1^2\lambda_2^2 (4\alpha_1\beta_1)(2\alpha_1 - 1)(2\beta_1 - 1) - \lambda_1^2\lambda_2^2 (16\alpha_1^2\beta_1^2)\\
     = &\  \lambda_1^2\lambda_2^2 (4\alpha_1\beta_1) (- 2\alpha_1 -2\beta_1 + 1)\ <\ 0,
\end{aligned}
\]
since $\alpha_1,\beta_1 \ge 2$.

\medskip
\noindent\emph{Coefficient of $x_1^{4\alpha_2 -2}x_2^{4\beta_2 - 2}$:}
\[
\begin{aligned}
  \,&  \lambda_1^2\lambda_2^2 (4\alpha_2\beta_2)(2\alpha_2 - 1)(2\beta_2 - 1) - \lambda_1^2\lambda_2^2 (16\alpha_2^2\beta_2^2)\\
     = &\  \lambda_1^2\lambda_2^2 (4\alpha_2\beta_2) (- 2\alpha_2 -2\beta_2 + 1)\ <\ 0,
\end{aligned}
\]
since $\alpha_2,\beta_2 \ge 2$.

\medskip
\noindent\emph{Coefficient of $x_1^{3\alpha_1 + \alpha_2 -2}x_2^{3\beta_1 + \beta_2  - 2}$:} writing $s^* = \lambda_1\lambda_2(\lambda_1^2 + \lambda_2^2 - c)$,
\[
\begin{aligned}
    &\ s^*(\alpha_1 + \alpha_2)(\alpha_1 + \alpha_2 - 1)(2\beta_1)(2\beta_1 - 1) +
     s^*(2\alpha_1)(2\alpha_1 - 1)(\beta_1 + \beta_2)(\beta_1 + \beta_2 - 1) \\
    &\quad - 2 s^*(2\alpha_1)(2\beta_1)(\alpha_1 + \alpha_2)(\beta_1 + \beta_2)\\
    = &\ s^*(2\alpha_2\beta_1 - 2\alpha_1\beta_2)^2 - s^*(\alpha_1 + \alpha_2)(2\beta_1)(\alpha_1 + \alpha_2 + 2\beta_1 - 1)\\
    &\quad - s^*(2\alpha_1)(\beta_1 + \beta_2)(2\alpha_1 + \beta_1 + \beta_2 - 1)\ <\ 0 .
    \end{aligned}
\]
The coefficient of $x_1^{\alpha_1 + 3\alpha_2 -2}x_2^{\beta_1 + 3\beta_2  - 2}$ is negative by the same computation.

\medskip
\noindent\emph{Coefficient of $x_1^{2\alpha_1 + 2\alpha_2 -2}x_2^{2\beta_1 + 2\beta_2  - 2}$:}
\[
\begin{aligned}
   & \lambda_1^2\lambda_2^2(2\alpha_2)(2\beta_1)(2\alpha_2 - 1)(2\beta_1 - 1) + \lambda_1^2\lambda_2^2(2\alpha_1)(2\alpha_1 - 1)(2\beta_2)(2\beta_2 - 1) \\
   & - 2\lambda_1^2\lambda_2^2 (2\alpha_1)(2\beta_1)(2\alpha_2)(2\beta_2) + (\lambda_1^2 + \lambda_2^2 - c)^2(\alpha_1 + \alpha_2)(\alpha_1 + \alpha_2 - 1)(\beta_1 + \beta_2)(\beta_1 + \beta_2 - 1) \\
   & - (\lambda_1^2 + \lambda_2^2 - c)^2(\alpha_1 + \alpha_2)^2(\beta_1 + \beta_2)^2 .
\end{aligned}
\]
Splitting this into two parts, the first,
\[
\lambda_1^2\lambda_2^2(2\alpha_2)(2\beta_1)(2\alpha_2 - 1)(2\beta_1 - 1) + \lambda_1^2\lambda_2^2(2\alpha_1)(2\alpha_1 - 1)(2\beta_2)(2\beta_2 - 1)
   - 2\lambda_1^2\lambda_2^2 (2\alpha_1)(2\beta_1)(2\alpha_2)(2\beta_2) < 0,
\]
is negative by the computation used for $x_1^{3\alpha_1 + \alpha_2 -2}x_2^{3\beta_1 + \beta_2 - 2}$, and the second,
\[
(\lambda_1^2 + \lambda_2^2 - c)^2\Big[(\alpha_1 + \alpha_2)(\alpha_1 + \alpha_2 - 1)(\beta_1 + \beta_2)(\beta_1 + \beta_2 - 1)
   - (\alpha_1 + \alpha_2)^2(\beta_1 + \beta_2)^2\Big] < 0,
\]
is negative by the computation used for $x_1^{4\alpha_1 -2}x_2^{4\beta_1 - 2}$.

Every collected coefficient is therefore negative, proving
\[
 H_{11}H_{22} - H_{12}^2 < 0
\]
and hence that polynomials of the form \eqref{eq:nonconvex-form} are non-convex, which completes the proof of Theorem~\ref{thm:nonconvex-2var}.
\end{proof}

\end{document}